\documentclass[12pt,reqno]{amsart}
\usepackage{amsmath,amssymb,amscd,graphicx}
\usepackage[all]{xy}
\usepackage[latin1]{inputenc}
\usepackage[T1]{fontenc}
\usepackage{hyperref}
\usepackage{ifthen}
\usepackage{fullpage}
\usepackage{subsupscripts}

\newboolean{workmode}
\setboolean{workmode}{false}

\newboolean{sections}
\setboolean{sections}{true}

\input xy
\xyoption{all}

\makeatletter
\def\th@plain{%
  \thm@notefont{}
  \itshape 
}
\def\th@definition{%
  \thm@notefont{}
  \normalfont 
}
\makeatother

\newcommand{\NN}{{\mathbb{N}}}  
\newcommand{\RR}{{\mathbb{R}}}  
\renewcommand{\SS}{{\mathbb{S}}} 
\newcommand{\ZZ}{{\mathbb{Z}}}  

\newcommand{\Ad}{{\operatorname{Ad}}}  
\newcommand{\Der}{{\operatorname{Der}}}  
\newcommand{\Diff}{{\operatorname{Diff}}}  
\newcommand{\ham}{{\operatorname{ham}}} 
\newcommand{\pr}{{\operatorname{pr}}} 
\newcommand{\Stab}{{\operatorname{Stab}}} 
\newcommand{\Vect}{{\operatorname{vect}}}  

\newcommand{\g}{{\mathfrak{g}}} 
\newcommand{\SO}{{\operatorname{SO}}}  

\newcommand{\CIN}{{C^\infty}}   
\newcommand{\hook}{{\lrcorner\,}} 
\newcommand{\pois}[2]{{\{#1,#2\}}}  
\newcommand{\red}[1]{{/\!/\!_#1\,}} 

\newcommand{\hypref}[2]{{\hyperref[#1]{#2~\ref{#1}}}}

\newcommand{\ifwork}[1]{\ifthenelse{\boolean{workmode}}{#1}{}}
\newcommand{\comment}[1]{}
\newcommand{\mute}[1]{}
\newcommand{\printname}[1]{}

\ifwork{
\renewcommand{\comment}[1]{{\marginpar{*}\ \scriptsize{#1}\ }}
\renewcommand{\mute}[1]{{\scriptsize \ #1\ }\marginpar{\scriptsize muted}}
\renewcommand{\printname}[1]
    {\smash{\makebox[0pt]{\hspace{-1.0in}\raisebox{8pt}{\tiny #1}}}}
}

\newcommand{\labell}[1] {\label{#1} \printname{#1}}

\newcommand{\ifsection}[2]{\ifthenelse{\boolean{sections}}{#1}{#2}}

\theoremstyle{plain}
\ifsection{
    \newtheorem{theorem}{Theorem}[section]
}
{
    \newtheorem{theorem}{Theorem}
}

\newtheorem{proposition}[theorem]{Proposition}

\newtheorem{corollary}[theorem]{Corollary}
\newtheorem{lemma}[theorem]{Lemma}

\theoremstyle{definition}
\newtheorem{definition}[theorem]{Definition}
\newtheorem{example}[theorem]{Example}
\newtheorem{remark}[theorem]{Remark}

\newtheorem{notation}[theorem]{Notation}

\newtheorem{question}[theorem]{Question}

\author{Jordan Watts}
\address{Department of Mathematics, University of Illinois at Urbana-Champaign, Urbana, Illinois, USA 61801}
\email{jawatts@illinois.edu}

\title{Differential Spaces, Vector Fields, and Orbit-Type Stratifications}
\date{\today}

\begin{document}

\begin{abstract}
Let $G$ be a Lie group, and let $(M,\omega)$ be a symplectic manifold.  If $G$ admits a Hamiltonian action on $(M,\omega)$ with momentum map $\mu$, then $M$, the zero-level set of $\mu$, the orbit space, and the corresponding symplectic quotient all have induced stratifications.  We push this setting into the language of differential spaces, and as a consequence we find that the stratifications are intrinsic to the ring of smooth functions on each space.
\end{abstract}

\maketitle

\tableofcontents

\section{Introduction}\labell{s:intro}

Let $G$ be a Lie group admitting a proper Hamiltonian group action on a symplectic manifold $(M,\omega)$ with momentum map $\mu:M\to\g^*$.  Let $Z:=\mu^{-1}(0)$.  Then it is known that the action induces stratifications of $M$ and $Z$, as well as the orbit space $M/G$ and the symplectic quotient $Z/G$ (see, for example, \cite{bierstone1}, \cite{bierstone2}, \cite{CS}, and \cite{DK}, as well as \cite{lerman-sjamaar}, \cite{BL}, and \cite{LW}).  These are the so-called orbit-type stratifications of each space.

Now, $(M,\omega)$ is naturally a Poisson manifold, and this structure induces a Poisson structure on $Z/G$, which in turn induces symplectic structures on each of the strata.  The purpose of this paper is to port the language of stratified spaces in this context into the language of differential structures and subcartesian spaces (see \cite{sniatycki}, \cite{lusala-sniatycki}, and \cite{sniatycki-book}).  As a result, we will see that the orbit-type stratifications are intrinsic to the rings of smooth functions (induced by the smooth manifold structure on $M$) on each respective space.  For the symplectic quotient, we need the Poisson structure as well in the case that $0\in\g^*$ is a critical value of $\mu$.  It remains an open problem whether, in the critical case, the orbit-type stratification on $Z/G$ is independent of the Poisson structure (see \hypref{t:regvalue}{Theorem} and \hypref{q:critvalue}{Question}).  Much of this work is already known (see \cite{sniatycki}, \cite{lusala-sniatycki}, \cite{sniatycki-book}, \cite{lerman-sjamaar}, \cite{BL}, and \cite{LW}), however, not all of it has been presented using differential structures; hence, in particular, we emphasise the intrinsicality of the stratifications to the smooth functions on $Z$ and $Z/G$.  Much of the material in this paper also appears in \cite{watts2}.

This paper is broken down as follows.  \hypref{s:preliminaries}{Section} is a comprehensive set of preliminaries necessary for the rest of the paper.  \hypref{s:zariski}{Section} reviews background on the Zariski tangent bundle on a subcartesian space.  \hypref{s:vectorfields}{Section}, \hypref{s:locallycomplete}{Section}, and \hypref{s:orbits}{Section} review the theory of vector fields on subcartesian spaces as developed by \'Sniatycki \cite{sniatycki}, \cite{sniatycki-book}.  \hypref{s:orbitaltangentbundle}{Section}, \hypref{s:liealgebra}{Section}, and \hypref{s:orbitalmaps}{Section} introduce a subcategory of subcartesian spaces (in particular, subcartesian spaces equipped with a family of vector fields) in which the context of a Hamiltonian group action sits naturally.

Finally, there has been much study of the symplectic quotient in the case that it is not a manifold; this is generally referred to as \emph{singular reduction}, and references besides the above include Arms-Cushman-Gotay \cite{ACG}, Guillemin-Ginzburg-Karshon \cite{GGK}, and Meinrenken-Sjamaar \cite{MS}.

The author would like to thank Yael Karshon for her enthusiasm, and both her and J\k{e}drzej \'Sniatycki for many important discussions.

\section{Preliminaries}\labell{s:preliminaries}
\subsection*{\emph{The Setting}}

Let $G$ be a compact Lie group acting smoothly on a connected smooth manifold $M$.

\begin{definition}
The \emph{orbit space} $M/G$ of this action is the set of equivalence classes given by the following equivalence relation on $M$: for $x,y\in M$, $x\sim y$ if $x$ and $y$ are in the same $G$-orbit; that is, if there exists $g\in G$ such that $g\cdot x=y$.  We equip $M/G$ with the quotient topology, which makes $M/G$ into a Hausdorff and locally compact space.
\end{definition}

Now assume $M$ is a symplectic manifold with symplectic form $\omega$.

\begin{definition}
A smooth $G$-action on $(M,\omega)$ is \emph{Hamiltonian} if the action preserves $\omega$ and there exists a smooth map $\mu:M\to\g^*$ (where $\g^*$ is the dual to the Lie algebra $\g$ of $G$) satisfying:
\begin{enumerate}
    \item $\mu$ is $G$-equivariant with respect to the coadjoint action of $G$ on $\g^*$,
    \item For any $\xi\in\g$, let $\xi_M$ be the vector field induced by $\xi$ on $M$: for any $x\in M$, $$\xi_M|_x:=\frac{d}{dt}\Big|_{t=0}\exp(t\xi)\cdot x,$$ and let $\mu^\xi(x):=\langle\mu(x),\xi\rangle$, where $\langle,\rangle$ is the pairing between $\g^*$ and $\g$.  Then $$\xi_M\hook\omega=-d\mu^\xi.$$
\end{enumerate}
We call $\mu$ a \emph{momentum map}.
\end{definition}

Let $Z:=\mu^{-1}(0)$.

\begin{theorem}[Marsden-Weinstein, Meyer]
If $0$ is a regular value of $\mu$, then $i:Z\hookrightarrow M$ is a smooth embedded $G$-invariant submanifold of $M$.  If in addition $G$ acts freely on $Z$, then the orbit space $M\red{0}G:=Z/G$ is a smooth manifold, $\pi_Z:Z\to M\red{0}G$ is a principal $G$-bundle, and $M\red{0}G$ admits a symplectic form $\omega_0$ satisfying $\pi_Z^*\omega_0=i^*\omega$.
\end{theorem}
\begin{proof}
(See \cite{MW} and \cite{meyer}.)
\end{proof}

\begin{definition}
$M\red{0}G$ above is called the \emph{symplectic reduced space} of the action, and $\omega_0$ the \emph{reduced symplectic form}.
\end{definition}

If $0$ is not a regular value of $\mu$, then $Z$ may not be a smooth submanifold, and thus $M\red{0}G$ need not be a smooth manifold.  The latter case may also occur if $G$ does not act freely on $Z$.  In such cases, $Z$ is still $G$-invariant, and the quotient space $Z/G$ equipped with the quotient topology is still Hausdorff and locally compact.  We thus obtain the following commutative diagram of continuous maps, where $j:Z/G\hookrightarrow M/G$ is the inclusion, and $i$ is smooth:

\begin{equation}\labell{d:main}
\xymatrix{
Z \ar[r]^{i} \ar[d]_{\pi_Z} & M \ar[d]^{\pi} \\
Z/G \ar[r]_{j} & M/G \\
}
\end{equation}

\subsection*{\emph{Differential Structures}}

\begin{definition}[Differential Space]
Let $X$ be a nonempty set.  A \emph{differential structure} on $X$ is a nonempty family $\mathcal{F}$ of functions into $\RR$, along with the weakest topology on $X$ for which every element of $\mathcal{F}$ is continuous, satisfying the following conditions.
\begin{enumerate}
\item (Smooth Compatibility) For any positive integer $k$, functions $f_1,...,f_k\in\mathcal{F}$, and $F\in\CIN(\RR^k)$, the composition $F(f_1,...,f_k)$ is contained in $\mathcal{F}$.
\item (Locality) Let $f:X\to\RR$ be a function such that for any $x\in X$ there exist an open neighbourhood $U\subseteq X$ of $x$ and a function $g\in\mathcal{F}$ satisfying $f|_U=g|_U$.  Then $f\in\mathcal{F}$.
\end{enumerate}
A set $X$ equipped with a differential structure $\mathcal{F}$ is called a \emph{differential space} and is denoted $(X,\mathcal{F})$.
\end{definition}

\begin{remark}\labell{r:diffstr}
\noindent
\begin{enumerate}
\item Let $X$ be a set and $\mathcal{F}$ a family of real-valued functions on it.  We will call the weakest topology on $X$ such that $\mathcal{F}$ is a set of continuous functions the \emph{topology induced} or \emph{generated} by $\mathcal{F}$, and denote it by $\mathcal{T}_\mathcal{F}$.  A subbasis for this topology is given by $$\{f^{-1}(I)~|~f\in\mathcal{F},~I\text{ is an open interval in $\RR$}\}.$$  In the case that $\mathcal{F}$ is a differential structure, by smooth compatibility and the facts that translation and rescaling are smooth, the subbasis is equal to $$\{f^{-1}((0,1))~|~f\in\mathcal{F}\}.$$  We will often refer to this as the \emph{subbasis induced} or \emph{generated} by $\mathcal{F}$.  Also, the basis comprised of finite intersections of elements of this subbasis we will refer to as the \emph{basis induced} or \emph{generated} by $\mathcal{F}$.  \comment{basis and subbasis remarks needed?}
\item The smooth compatibility condition of a differential structure guarantees that $\mathcal{F}$ is a commutative $\RR$-algebra under pointwise addition and multiplication.
\item The locality condition indicates that a differential structure $\mathcal{F}$ on $X$ induces a sheaf of functions: for any open subset $U$ of $X$, define $\mathcal{F}(U)$ to be all functions $f:U\to\RR$ such that if $x\in U$, then there exist an open neighbourhood $V\subseteq U$ of $x$ and a function $g\in\mathcal{F}$ such that $$g|_V=f|_V.$$
\end{enumerate}
\end{remark}

\begin{example}[Manifolds]
A manifold $M$ comes equipped with the differential structure given by its smooth functions $\CIN(M)$.
\end{example}

\begin{definition}[Functionally Smooth Maps]
Let $(X,\mathcal{F}_X)$ and $(Y,\mathcal{F}_Y)$ be two differential spaces.  A map $F:X\to Y$ is \emph{functionally smooth} if $F^*\mathcal{F}_Y\subseteq\mathcal{F}_X$.  $F$ is called a \emph{functional diffeomorphism} if it is functionally smooth and has a functionally smooth inverse.  Denote the set of functionally smooth maps between $X$ and $Y$ by $\CIN(X,Y)$.
\end{definition}

\begin{remark}
Note that in the literature, differential structures, differential spaces, and functionally smooth maps are sometimes called \emph{Sikorski structures}, \emph{Sikorski spaces}, and \emph{Sikorski smooth} maps, respectively.  See, for example, \cite{stacey}.
\end{remark}

\begin{remark}
A functionally smooth map is continuous with respect to the topologies induced by the differential structures.
\end{remark}

\begin{example}[Smooth Maps Between Manifolds]
Given two manifolds $M$ and $N$, the functionally smooth maps between $M$ and $N$ are exactly the usual smooth maps $\CIN(M,N)$.
\end{example}

\begin{remark}
Differential spaces along with functionally smooth maps form a category.
\end{remark}

Let $X$ be a set, and let $\mathcal{Q}$ be a family of real-valued functions on $X$.  Equip $X$ with the topology induced by $\mathcal{Q}$.  Define a family $\mathcal{F}$ of real-valued functions on $X$ as follows. $f\in\mathcal{F}$ if for any $x\in X$ there exists an open neighbourhood $U\subseteq X$ of $x$, functions $q_1,...,q_k\in\mathcal{Q}$, and a function $F\in\CIN(\RR^k)$ satisfying $$f|_U=F(q_1,...,q_k)|_U.$$  (This is just the set of global sections of the sheafification of $\mathcal{Q}$.)

\begin{lemma}\labell{l:topologies}
The two topologies $\mathcal{T}_\mathcal{Q}$ and $\mathcal{T}_\mathcal{F}$ are equal.
\end{lemma}

\begin{proof}
Since $\mathcal{Q}\subseteq\mathcal{F}$, we have that the subbasis induced by $\mathcal{Q}$ is contained in the subbasis induced by $\mathcal{F}$, and so $\mathcal{T}_\mathcal{Q}\subseteq\mathcal{T}_\mathcal{F}$.  We now wish to show the opposite containment.\\

Fix $f\in\mathcal{F}$ and $x\in X$.  Let $I\subset\RR$ be an open interval containing $f(x)$.  We wish to find a set $W\in\mathcal{T}_\mathcal{Q}$ containing $x$ and contained in $f^{-1}(I)$.  By definition of $\mathcal{F}$, there is some set $U\in\mathcal{T}_\mathcal{Q}$ containing $x$, functions $q_1,...,q_k\in\mathcal{Q}$, and $F\in\CIN(\RR^k)$ such that $f|_U=F(q_1,...,q_k)|_U$.  Let $y=(q_1,...,q_k)(x)$, and let $B=\prod_{i=1}^k(a_i,b_i)$ be an open box containing $y$ and contained in $F^{-1}(I)$.  Then, $(q_1,...,q_k)^{-1}(B)\cap U$ is a set contained in $f^{-1}(I)\cap U\subseteq f^{-1}(I)$.  But, $$(q_1,...,q_k)^{-1}(B)=q_1^{-1}(\pr_1(B))\cap...\cap q_k^{-1}(\pr_k(B)),$$ where $\pr_i$ is the $i$th projection. This intersection is a finite intersection of open sets in $\mathcal{T}_\mathcal{Q}$.  Hence, $(q_1,...,q_k)^{-1}(B)\cap U$ is an open set in $\mathcal{T}_\mathcal{Q}$ containing $x$ and contained in $f^{-1}(I)$.  So let $W:=(q_1,...,q_k)^{-1}(B)\cap U$.
\end{proof}

\begin{proposition}
$(X,\mathcal{F})$ is a differential space.
\end{proposition}

\begin{proof}
First, we show smooth compatibility.  Let $f_1,...,f_k\in\mathcal{F}$ and $F\in\CIN(\RR^k)$.  Then, we want to show $F(f_1,...,f_k)\in\mathcal{F}$.  Fix $x\in X$.  Then for each $i=1,...,k$ there exist an open neighbourhood $U_i$ of $x$, $q_i^1,...,q_i^{m_i}\in\mathcal{Q}$ and $F_i\in\CIN(\RR^{m_i})$ such that $f_i|_{U_i}=F_i(q_i^1,...,q_i^{m_i})|_{U_i}$.  Let $U$ be the intersection of the neighbourhoods $U_i$, which itself is an open neighbourhood of $x$.  Then, $$F(f_1,...,f_k)|_U=F(F_1(q_1^1,...,q_1^{m_1}),...,F_k(q_k^1,...,q_k^{m_k}))|_U.$$ Let $N:=m_1+...+m_k$.  Define $\tilde{F}\in\CIN(\RR^N)$ by $$\tilde{F}(x^1,...,x^N)=F(F_1(x^1,...,x^{m_1}),F_2(x^{m_1+1},...,x^{m_1+m_2}),...,F_k(x^{m_1+...+m_{k-1}+1},...,x^N)).$$ Then $$F(f_1,...,f_k)|_U=\tilde{F}(q_1^1,...q_1^{m_1},q_2^1,...q_2^{m_2},...,q_k^1,...,q_k^{m_k})|_U.$$ By definition of $\mathcal{F}$, we have $F(f_1,...,f_k)\in\mathcal{F}$.\\

Next, we show locality.  Let $f:X\to\RR$ be a function with the property that for every $x\in X$ there is an open neighbourhood $U$ of $x$ and a function $g\in\mathcal{F}$ such that $f|_U=g|_U$.  Fix $x$, and let $U$ and $g$ satisfy this property.  Shrinking $U$ if necessary, there exist $q_1,...,q_k\in\mathcal{Q}$ and $F\in\CIN(\RR^k)$ such that $g|_U=F(q_1,...,q_k)|_U$.  Hence, $f|_U=F(q_1,...,q_k)|_U$.  Since this is true at each $x\in X$, by definition, $f\in\mathcal{F}$.  This completes the proof.
\end{proof}

\begin{definition}[Generated Differential Structures]
We say that the differential structure $\mathcal{F}$ above is \emph{generated} by $\mathcal{Q}$.
\end{definition}

\begin{lemma}
Let $(X,\mathcal{F})$ be a differential space.  Then for any subset $Y\subseteq X$, the subspace topology on $Y$ is the weakest topology for which the restrictions of $\mathcal{F}$ to $Y$ are continuous.
\end{lemma}

\begin{proof}
We first set some notation.  Let $\mathcal{T}_Y$ be the subspace topology on $Y$, and let $\mathcal{G}$ be all restrictions of functions in $\mathcal{F}$ to $Y$.\\

Fix $U\in\mathcal{T}_Y$ and $x\in U$.  We will show that there exists a basic open set $W$ in $\mathcal{T}_\mathcal{G}$ such that $x\in W\subseteq U$.  By definition of the subspace topology on $Y$, there exists an open set $V\in\mathcal{T}_\mathcal{F}$ such that $$U=V\cap Y.$$ There exist $f_1,...,f_k\in\mathcal{F}$ such that $$\tilde{W}:=\bigcap_{i=1}^kf_i^{-1}((0,1))$$ is a basic open set of $X$ containing $x$ and contained in $V$.  Define $W:=\tilde{W}\cap Y$.  Then,
\begin{align*}
W=&~\bigcap_{i=1}^kf_i^{-1}((0,1))\cap Y\\
=&~\bigcap_{i=1}^k(f_i|_Y)^{-1}((0,1)).
\end{align*}
But $f_i|_Y\in\mathcal{G}$, and so $W$ is a basic open set in $\mathcal{T}_{\mathcal{G}}$ that contains $x$ and is contained in $U$.\\

Next, we show that for any $U\in\mathcal{T}_\mathcal{G}$, $U$ is in fact in the subspace topology.  It is sufficient to show this for any basic open set $U$, in the basis generated by $\mathcal{G}$.  To this end, fix a basic open set $U\in\mathcal{T}_\mathcal{G}$ and $x\in U$.  There exist $g_1,...,g_k\in\mathcal{G}$ such that $$U=\bigcap_{i=1}^kg_i^{-1}((0,1)).$$  But then there exist $f_1,...,f_k\in\mathcal{F}$ such that for each $i=1,...,k$ we have $g_i=f_i|_Y$.  Then, $$U=\bigcap_{i=1}^kf_i^{-1}((0,1))\cap Y.$$  Since $\bigcap_{i=1}^kf_i^{-1}((0,1))$ is open on $X$, we have that $U$ is open in the subspace topology on $Y$.  We have shown that the subspace topology on $Y$ and the topology generated by restrictions of functions in $\mathcal{F}$ to $Y$ are one and the same.
\end{proof}

The above lemma allows us to make the following definition.

\begin{definition}[Differential Subspace]
Let $(X,\mathcal{F})$ be a differential space, and let $Y\subseteq X$ be any subset.  Then $Y$, with the subspace topology, acquires a differential structure $\mathcal{F}_Y$ generated by restrictions of functions in $\mathcal{F}$ to $Y$.  That is, $f\in\mathcal{F}_Y$ if and only if for every $x\in Y$ there is an open neighbourhood $U\subseteq X$ of $x$ and a function $\tilde{f}\in\mathcal{F}$ such that $$f|_{U\cap Y}=\tilde{f}|_{U\cap Y}.$$  We call $(Y,\mathcal{F}_Y)$ a \emph{differential subspace} of $X$.
\end{definition}

\begin{definition}[Product Differential Structure]
Let $(X,\mathcal{F})$ and $(Y,\mathcal{G})$ be two differential spaces.  The \emph{product differential space} $(X\times Y,\mathcal{F}\times\mathcal{G})$ is given by the set $X\times Y$ equipped with the differential structure $\mathcal{F}\times\mathcal{G}$, generated by functions of the form $f\circ\pr_X$ for $f\in\mathcal{F}$, and $g\circ\pr_Y$ for $g\in\mathcal{G}$.  Here, $\pr_X$ and $\pr_Y$ are the projections onto $X$ and $Y$, respectively.  In particular, the projection maps are functionally smooth.
\end{definition}

\begin{definition}[Quotient Differential Structure]
Let $(X,\mathcal{F})$ be a differential space, and let $\sim$ be an equivalence relation on $X$.  Then $X/\!\sim$ obtains a differential structure, called the \emph{quotient differential structure}, $\mathcal{G}=\{f:X/\!\sim\;\to\RR~|~\pi^*f\in\mathcal{F}\}$ where $\pi:X\to X/\!\sim$ is the quotient map.
\end{definition}

\begin{remark}
The quotient map $\pi$ above is smooth by definition.  Also,  we do not endow the set $X/\!\sim$ above with the quotient topology.  In general, the topology on $X/\!\sim$ induced by  $\mathcal{G}$ and the quotient topology do not match (see the following example).  In fact, the induced topology is contained in the quotient topology.
\end{remark}

\begin{example}[Quotient Topology Does Not Work]
We give an example to illustrate the issue with topologies mentioned in the above remark.  Consider the quotient space $\RR/I$, where $I$ is the open interval $(0,1)$.  By definition of the quotient topology, letting $\pi$ be the quotient map, we have that $\pi((0,1))$ is a one-point set that is open. $f$ is in the quotient differential structure if its pullback by $\pi$ is in $\CIN(\RR)$.  In this case, $\pi^*f$ is constant on $(0,1)$. But since level sets are closed, we have that $\pi^*f$ is constant on $[0,1]$.  Thus, $f$ is constant on the three-point set $\{\pi(0),\pi(1)\}\cup\pi((0,1))$.  Thus, the pre-image of any open interval of $\RR$ by any function in the quotient differential structure will never be included in the one-point set $\pi((0,1))$.  Thus, the quotient topology is strictly stronger than the topology induced by the quotient differential structure.
\end{example}

\begin{definition}[Subcartesian Space]\labell{d:subcart}
A \emph{subcartesian space} is a paracompact, second-countable, Hausdorff differential space $(S,\CIN(S))$ where for each $x\in S$ there is an open neighbourhood $U\subseteq S$ of $x$, $n\in\NN$, and a diffeomorphism $\varphi:U\to\tilde{U}\subseteq\RR^n$, called a \emph{chart}, onto a differential subspace $\tilde{U}$ of $\RR^n$.  Unless otherwise it is unclear, we shall henceforth call functionally smooth maps between subcartesian spaces simply \emph{smooth}.
\end{definition}

\begin{remark} \noindent
\begin{enumerate}
\item Subcartesian spaces, along with smooth maps between them, form a full subcategory of the category of differential spaces.
\item A subcartesian space admits smooth partitions of unity (see \cite{marshall}).
\item For any subset $A\subseteq\RR^n$, define $\mathfrak{n}(A)$ to be the ideal of all smooth functions on $\RR^n$ whose restrictions to $A$ are identically zero.  Let $S$ be a subcartesian space. Then, for each chart $\varphi:U\to\tilde{U}\subseteq\RR^n$, the set of restrictions of functions in $\CIN(\RR^n)$ to $\tilde{U}$ is isomorphic as an $\RR$-algebra to $\CIN(\RR^n)/\mathfrak{n}(\tilde{U})$.  We thus have $\varphi^*\CIN(\RR^n)\cong\CIN(\RR^n)/\mathfrak{n}(\tilde{U})$ as $\RR$-algebras.
\end{enumerate}
\end{remark}

\begin{proposition}[Closed Differential Subspaces of Subcartesian Spaces]\labell{p:closedsubset}
If $R$ is a closed differential subspace of a subcartesian space $S$, then $\CIN(R)=\CIN(S)|_R$, the restrictions of functions in $\CIN(S)$ to $R$.
\end{proposition}

\begin{proof}
It is clear that $\CIN(S)|_R\subseteq\CIN(R)$.  To show the opposite inclusion, fix $f\in\CIN(R)$.  By definition of $\CIN(R)$, we can find an open covering $\{U_\alpha\}_{\alpha\in A}$ of $R$ such that for each $\alpha$, there is a function $g_\alpha\in\CIN(S)$ satisfying $$g_\alpha|_{U_\alpha}=f|_{U_\alpha}.$$  Let $B=\{0\}\cup A$ (assume here that $A$ does not include $0$).  For each $\alpha\in A$, let $V_\alpha$ be an open subset of $S$ such that $U_\alpha=R\cap V_\alpha$.  Let $V_0$ be the complement of $R$ in $S$ and define $g_0:=0$.  Define $\{\zeta_\beta\}_{\beta\in B}$ to be a partition of unity subordinate to $\{V_\beta\}_{\beta\in B}$.  Let $\tilde{g}:=\sum_{\beta\in B}\zeta_\beta g_\beta$.  Then
\begin{align*}
\tilde{g}|_R=&~\sum_{\beta\in B}\zeta_\beta|_R g_\beta\\
=&~\sum_{\beta\in B}\zeta_\beta f\\
=&~f.
\end{align*}
\end{proof}

Let $G$ be a compact Lie group acting smoothly on a manifold $M$, and let $\pi:M\to M/G$ be the quotient map.  Equip the geometric quotient $M/G$ with the quotient differential structure.  Note that $\pi^*:\CIN(M/G)\to\CIN(M)^G$ is an isomorphism of $\RR$-algebras, where $\CIN(M)^G$ is the algebra of $G$-invariant smooth functions.

\begin{theorem}[Quotients by Compact Group Actions are Subcartesian]\labell{t:quotsubc}
If $G$ is a compact Lie group acting on a manifold $M$, then $M/G$ is a subcartesian space whose topology matches the quotient topology induced by $\pi$.
\end{theorem}

\begin{proof}
The fact that $M/G$ equipped with the quotient differential structure is a subcartesian space is proven by Schwarz in \cite{schwarz1}.  That the quotient topology and the induced topology from $\CIN(M/G)$ are the same is shown by Cushman-\'Sniatycki in \cite{CS}.
\end{proof}

\begin{remark}
The above theorem extends to proper group actions using the Slice Theorem of Palais \cite{palais}.
\end{remark}

We again come back to \hypref{d:main}{Diagram}.  $Z$ comes equipped with a differential structure $\CIN(Z)$ induced by $M$.  In particular, since $Z$ is closed, by \hypref{p:closedsubset}{Proposition} we have that $i^*\CIN(M)=\CIN(Z)$.  Consequently, $Z/G$ is a closed differential subspace of $M/G$.

\begin{theorem}[Symplectic Quotients are Subcartesian]\labell{t:CINZ}
$Z/G$ as a subspace of $M/G$ is a subcartesian space.  Moreover, its subspace differential structure is equal to the quotient differential structure obtained from $Z$.
\end{theorem}

\begin{proof}
Note that $Z/G$ is a closed subset of $M/G$ (and hence is subcartesian), and so $\CIN(Z/G)=\CIN(M/G)|_{Z/G}$ by \hypref{p:closedsubset}{Proposition}.  We now show that $\pi_Z$ is smooth.  Let $f\in\CIN(Z/G)$.  Then there exists $g\in\CIN(M/G)$ such that $f=j^*g$.  Let $\tilde{g}=\pi^*g\in\CIN(M)^G$.  Let $\tilde{f}=i^*\tilde{g}\in\CIN(Z)^G$.  Then, $\tilde{f}=\pi_Z^*f$.\\

Next, since $\pi_Z$ is surjective, $\pi_Z^*$ is injective.  To show that $\pi_Z^*$ is surjective onto $\CIN(Z)^G$, fix $\tilde{f}\in\CIN(Z)^G$. Since $Z$ is closed, applying \hypref{p:closedsubset}{Proposition} once again, there exists $\tilde{g}\in\CIN(M)$ such that $\tilde{f}=i^*\tilde{g}$.  Averaging over $G$, we may assume that $\tilde{g}$ is $G$-invariant.  Thus, there exists $g\in\CIN(M/G)$ such that $\pi^*g=\tilde{g}$.  Thus, $f=j^*g\in\CIN(Z/G)$, and $\pi_Z^*f=\tilde{f}$.  We get that $\pi_Z^*:\CIN(Z/G)\to\CIN(Z)^G$ is an isomorphism of $\RR$-algebras.
\end{proof}

\begin{remark}
The smooth structure $\CIN(Z/G)$ is equal to a smooth structure on $Z/G$ introduced by Arms, Cushman and Gotay in \cite{ACG}.  The isomorphism $\pi_Z^*:\CIN(Z/G)\to\CIN(Z)^G$ is in fact the definition of the latter.
\end{remark}

\subsection*{\emph{Stratified Spaces}}
Unfortunately, in the literature, there are many definitions of stratified spaces, not all of which are equivalent (see \cite{sniatycki}, \cite{pflaum}, \cite{GM}).  For our purposes, we start with a topological definition, following closely the terminology used in \cite{lerman-sjamaar}.  We then transport these concepts into the differential space category, following closely concepts introduced in \cite{sniatycki} and \cite{lusala-sniatycki}.\\

Let $X$ be a Hausdorff, paracompact topological space, and $(A,\leq)$ a partially ordered set.

\begin{definition}[Decomposed Space]
A \emph{decomposition} of $X$ with respect to $(A,\leq)$ is a locally finite partition of $X$, denoted by $\mathcal{P}$, into disjoint, connected, locally closed (topological) manifolds $S_i$, called \emph{pieces} such that the set $\mathcal{P}$ is indexed by $A$, and $i\leq j$ if and only if $S_i\subseteq\overline{S}_j$ if and only if $S_i\cap\overline{S}_j\neq\emptyset$.  The \emph{dimension} of $X$, denoted $\dim(X)$, is the supremum over $A$ of the dimensions of the pieces.  $X$ equipped with a decomposition $\mathcal{P}$ will be referred to as a \emph{decomposed space}, denoted $(X,\mathcal{P})$.  Often we will drop the notation $\mathcal{P}$ when the decomposition has been made clear.
\end{definition}

\begin{remark}
We will only consider decomposed spaces in which the pieces are finite-dimensional.
\end{remark}

\begin{definition}[Depth]
Let $(X,\mathcal{P})$ be a decomposed space, and fix a piece $S\in\mathcal{P}$.  The \emph{depth} of $S$, denoted $\operatorname{depth}_X(S)$ is defined as $$\operatorname{depth}_X(S):=\sup\{n\in\NN~|~S=S_{a_0}\subsetneq S_{a_1}\subsetneq...\subsetneq S_{a_n}\}$$ where each $S_{a_i}\in\mathcal{P}$.  Note the \emph{strict} inclusions in the definition. The \emph{depth} of $X$ is given by $$\operatorname{depth}(X):=\sup\{\operatorname{depth}_X(S_a)~|~a\in A\}.$$
\end{definition}

A stratified space is a decomposed space in which the pieces  fit together in a specific way.  Note that the following definition is recursive (in particular, $F$ will have a smaller depth than $S$).

\begin{definition}[Stratified Space]
A decomposed space $X$ is a \emph{stratified space} if the pieces of $X$, called \emph{strata}, satisfy the following condition.

\begin{itemize}\labell{cond:localtriv}
\item (Local Triviality) For every $x\in X$, there is an open neighbourhood $U\subseteq X$ of $x$, a stratified space $F$ with a distinguished point $o\in F$ such that $\{o\}$ is a stratum of $F$, and a homeomorphism $\varphi:U\to(S\cap U)\times F$ where $S$ is the stratum of $X$ containing $x$.  $\varphi$ is required to satisfy $$\varphi(s)=(s,o)$$ for each $s\in S\cap U$, and to map strata into strata.
\end{itemize}
\end{definition}

\begin{remark}
The above local triviality condition is often written in the literature using cones over stratified spaces instead of $F$.  However, it will be easier to transport the definition we use above into the smooth category.
\end{remark}

\begin{example}\labell{x:square}
Consider the square $[0,1]\times[0,1]$.  The partition given by $$\mathcal{P}=\Big\{\{(0,0)\},\;\{(0,1)\},\;\{(1,0)\},\;\{(1,1)\},\;(0,1)\times\{0\},$$$$
(0,1)\times\{1\},\;\{0\}\times(0,1),\;\{1\}\times(0,1),\;(0,1)^2\Big\}$$ makes the square into a stratified space.
\end{example}

\begin{definition}[Smooth Decomposed Space]
A \emph{smooth decomposed space} is a triple $(X,\mathcal{F},\mathcal{P})$ where $(X,\mathcal{F})$ is a differential space, and $(X,\mathcal{P})$ is a decomposed space with respect to the topology induced by $\mathcal{F}$.  We require that for each piece $S\in\mathcal{P}$, the inclusion map $i_S:S\to X$ induces a smooth manifold structure on $S$.
\end{definition}

\begin{definition}[Smooth Stratified Space]
A smooth decomposed space $X$ is a \emph{smooth stratified space} if the pieces of the decomposition satisfy the following condition.
\begin{itemize}
\item (Smooth Local Triviality) For every $x\in X$, there is an open neighbourhood $U\subseteq X$ of $x$, a smooth stratified space $F$ with a distinguished point $o\in F$ such that $\{o\}$ is a stratum of $F$, and a diffeomorphism $\varphi:U\to(S\cap U)\times F$ where $S$ is the stratum of $X$ containing $x$.  $\varphi$ is required to satisfy $$\varphi(s)=(s,o)$$ for each $s\in S\cap U$, and to map strata into strata.
\end{itemize}
Again, the pieces in the decomposition in this case are called \emph{strata}.
\end{definition}

For our purposes, we will always assume that the differential structure on a smooth stratified space is subcartesian.

\begin{definition}[Smooth Stratified Map]
Let $X$ and $Y$ be smooth stratified spaces, and let $F:X\to Y$ be a smooth map.  Then $F$ is \emph{stratified} if for each stratum $S$ of $X$, $F(S)\subset T$ for some stratum $T$ of $Y$.
\end{definition}

\begin{remark}
Smooth stratified spaces, along with stratified maps, form a category.
\end{remark}

Let $G$ be a compact Lie group acting on a (smooth) manifold $M$.  Let $H$ be a closed subgroup of $G$, and let $M_{(H)}$ be the set of all points in $M$ whose stabiliser is a conjugate of $H$.  Then, $M$ is the disjoint union of the sets $M_{(H)}$ as $H$ runs over closed subgroups of $G$.  The quotient map $\pi:M\to M/G$ partitions $M/G$ into sets $(M/G)_{(H)}:=\pi(M_{(H)})$ as $H$ runs over closed subgroups of $G$.

\begin{theorem}[Orbit-Type Stratification]\labell{t:pot}
\noindent
\begin{enumerate}
\item The partitions on $M$ and $M/G$ defined above yield smooth stratifications with respect to the smooth structures $\CIN(M)$ and $\CIN(M/G)$, respectively, whose strata are given by connected components of the sets $M_{(H)}$ and $(M/G)_{(H)}$.
\item Each subset $M_{(H)}$ is a $G$-invariant submanifold of $M$.  If $M$ and $G$ are connected, then each stratum in the stratification on $M$ is $G$-invariant.
\item If $M$ is connected, then there exists a closed subgroup $K$ of $G$ such that the strata contained in $M_{(K)}$ form an open dense subset of $M$, and hence $(M/G)_{(K)}$ is an open dense subset of $M/G$.
\item The orbit map $\pi:M\to M/G$ is stratified with respect to the stratifications described above.
\end{enumerate}
\end{theorem}
\begin{proof}
The last statement above is clear by definition of the stratifications. See \cite{DK} and \cite{CS} for the first three statements.
\end{proof}

\begin{definition}\labell{d:symplquotstrat}
We call the above stratifications \emph{orbit-type stratifications} of each respective space.
\end{definition}

We now return to the setting above in \hypref{d:main}{Diagram}.  For each closed subgroup $H$ of $G$, let $Z_{(H)}:=M_{(H)}\cap Z$.  Note that this is a $G$-invariant subset of $Z$ since both $Z$ and $M_{(H)}$ are invariant.  Let $(Z/G)_{(H)}:=\pi(Z_{(H)})$.  For each nonempty such subset, let $\pi_{(H)}:=\pi|_{Z_{(H)}}$ and $i_{(H)}:=i|_{Z_{(H)}}$.  Finally, let $\pi_Z:=\pi|_Z$.

\begin{theorem}[Orbit-Type Stratification (Hamiltonian Version)]\labell{t:poth}
\noindent
\begin{enumerate}
\item The partitions on $Z$ and $Z/G$ defined above yield smooth stratifications with respect to $\CIN(Z)$ and $\CIN(Z/G)$, respectively, whose strata are given by connected components of the sets $Z_{(H)}$ and $(Z/G)_{(H)}$.
\item Each subset $Z_{(H)}$ is a $G$-invariant submanifold of $M$.
\item If $M$ is connected and $\mu$ is a proper map, then there exists a closed subgroup $K$ of $G$ such that the strata contained in $Z_{(K)}$ form an open dense subset of $Z$, and hence $(Z/G)_{(K)}$ is an open dense subset of $Z/G$.
\item The orbit map $\pi_Z:Z\to Z/G$ along with the inclusions $i$ and $j$ are stratified with respect to the stratifications described above.
\item For each nonempty stratum $(Z/G)_{(H)}$, there exists a symplectic form $\omega_{(H)}$ on $(Z/G)_{(H)}$ such that $\pi_{(H)}^*\omega_{(H)}=i_{(H)}^*\omega$.
\end{enumerate}
\end{theorem}

\begin{proof}
See \cite{lerman-sjamaar}.
\end{proof}

\begin{definition}
The stratifications defined on $Z$ and $Z/G$ above are called \emph{orbit-type stratifications} of each space.
\end{definition}

\begin{remark}
Note that $$(Z/G)_{(H)}=Z_{(H)}/G=(Z/G)\cap(M/G)_{(H)}.$$
\end{remark}

To summarise, \hypref{d:main}{Diagram} sits in the category of smooth stratified spaces.

\subsection*{\emph{Basic Forms}}

We continue with the setting given by \hypref{d:main}{Diagram}.

\begin{definition}
A differential form $\mu\in\Omega^k(M)$ is \emph{basic} if it is
\begin{enumerate}
\item $G$-\emph{invariant}: for any $g\in G$, $g^*\mu=\mu$, and
\item \emph{horizontal}: for any $\xi\in\g$, $$\xi_M\hook\mu=0.$$
\end{enumerate}
Denote the set of all basic $k$-forms on $M$ by $\Omega^k_{basic}(M)$.
\end{definition}

\begin{theorem}[Koszul]\labell{t:basic}
$(\Omega^*_{basic}(M),d)$ forms a subcomplex of $(\Omega^*(M),d)$, and the corresponding cohomology $H_{basic}^*(M)$ is isomorphic to the singular cohomology $H^*(M/G)$ with real coefficients.
\end{theorem}

\begin{proof}
See \cite{koszul}.
\end{proof}

\begin{remark}\labell{r:basic}
If $M/G$ is a smooth manifold, then for each $\mu\in\Omega^k_{basic}(M)$, there exists a unique $\eta\in\Omega^k(M/G)$ such that $\pi^*\eta=\mu$.  Thus, in light of the above theorem and the de Rham theorem, $\pi^*:(\Omega^*(M/G),d)\to(\Omega^*_{basic}(M),d)$ is an isomorphism of complexes.
\end{remark}

\subsection*{\emph{Left-Invariant and Hamiltonian Vector Fields}}

We continue to be in the setting of a Hamiltonian group action (\hypref{d:main}{Diagram}).

\begin{definition}
A vector field $X$ on $M$ is \emph{left-invariant} if for any $g\in G$, we have $g_*X=X$.  Denote the set of left-invariant vector fields by $\Vect(M)^G$.  A vector field $X$ on $M$ is \emph{Hamiltonian} if there exists a function $f\in\CIN(M)$ satisfying $X\hook\omega=-df$.  In this case, it is customary to denote $X$ as $X_f$.  This induces an $\RR$-linear map $\CIN(M)\to\Vect(M)$.  Denote the image $\ham(M)$, and the left-invariant Hamiltonian vector fields by $\ham(M)^G$.
\end{definition}

\begin{remark}
$\Vect(M)^G$, $\ham(M)$ and $\ham(M)^G$ are all Lie subalgebras of $\Vect(M)$ under the commutator bracket.  Note also that for any $X\in\Vect(M)^G$, the local flow $(t,x)\mapsto\exp(tX)(x)$ of $X$ is $G$-equivariant: for any $g\in G$, $$g\cdot\exp(tX)(x)=\exp(tX)(g\cdot x).$$
\end{remark}

\begin{example}
Let $\xi\in\g$.  Then the induced vector field $\xi_M$ is Hamiltonian.
\end{example}

\begin{proposition}\labell{p:leftinvtham}
\noindent
\begin{enumerate}
\item If $X$ is a left-invariant vector field, then for any $H\leq G$ such that $M_{(H)}$ is nonempty, $X$ is tangent to $M_{(H)}$, and so restricts to a vector field on $M_{(H)}$.
\item If $X$ is a left-invariant Hamiltonian vector field, then $X$ is tangent to the $G$-manifold $Z_{(H)}$, and so $i_{(H)}^*X$ is well-defined.
\end{enumerate}
\end{proposition}
\begin{proof}
\noindent
\begin{enumerate}
\item Let $\psi_t$ be the flow of $X$.  Then $\psi_t$ is $G$-equivariant: if $x\in M$ and $g\in G$, then $\psi_t(g\cdot x)=g\cdot\psi_t(x)$ for all $t$ in the flow domain.  If $g\in\Stab(x)=:H$, then $\psi_t(x)=\psi_t(g\cdot x)=g\cdot\psi_t(x)$, and so $\psi_t$ preserves stabilisers along its trajectories.  Thus, these trajectories remain in $M_{(H)}$, and so $X$ is tangent to $M_{(H)}$.
\item Using the first statement of the proposition, it is enough to show that $X\mu^\xi=0$ for any $\xi\in\g$.  Let $f\in\CIN(M)$ such that $X=X_f$.  By averaging over $G$, we can choose $f$ such that it is $G$-invariant. Thus,
    \begin{align*}
    X\mu^\xi=&~d\mu^\xi(X)\\
    =&~\omega(X,\xi_M)\\
    =&~-df(\xi_M)=0.
    \end{align*}
    This completes the proof.
\end{enumerate}
\end{proof}

\begin{lemma}\labell{l:jpois1}
For any $H\leq G$ such that $Z_{(H)}$ is nonempty, there is a Lie algebra homomorphism $\ham(M)^G\to\ham((Z/G)_{(H)})$ sending $X_f$ to $X_h$ where $h=j^*((\pi^*)^{-1}f)|_{(Z/G)_{(H)}}$ and $(\pi^*)^{-1}:\CIN(M)^G\to\CIN(M/G)$.
\end{lemma}
\begin{proof}
Fix $X_f\in\ham(M)^G$ where $f\in\CIN(M)^G$ and $H\leq G$ such that $Z_{(H)}$ is nonempty. Then there exists $g\in\CIN(M/G)$ such that $\pi^*g=f$.  By \hypref{p:leftinvtham}{Proposition}, $X_f$ restricts to a $G$-invariant vector field $(X_f)|_{Z_{(H)}}$ on $Z_{(H)}$.  This descends via $\pi_{(H)}$ to a vector field $Y$ on $(Z/G)_{(H)}$.  We claim that $Y$ is the Hamiltonian vector field of the smooth function $j^*g|_{(Z/G)_{(H)}}$.  Indeed,
\begin{align*}
\pi_{(H)}^*(Y\hook\omega_{(H)})=&~X_f|_{Z_{(H)}}\hook i_{(H)}^*\omega\\
=&~i_{(H)}^*(X_f\hook\omega)\\
=&~i_{(H)}^*(-df)\\
=&~(\pi\circ i_{(H)})^*(-dg)\\
=&~(j\circ\pi_{(H)})^*(-dg)\\
=&~\pi_{(H)}^*(-dj^*g|_{(Z/G)_{(H)}}).
\end{align*}

By \hypref{t:basic}{Theorem} and \hypref{r:basic}{Remark}, $\pi_{(H)}^*$ is an isomorphism onto its image, and so we conclude that $Y\hook\omega_{(H)}=-dj^*g$.  Letting $h=j^*g|_{(Z/G)_{(H)}}$, we thus have $Y=X_{h}$.
\end{proof}

For any $H\leq G$ such that $Z_{(H)}$ is nonempty, let $X_f,X_g\in\ham(M)^G$.  There exist $Y_1,Y_2\in\ham((Z/G)_{(H)})^G$ satisfying \hypref{l:jpois1}{Lemma}.  We have the following second lemma.

\begin{lemma}\labell{l:jpois2}
$i_{(H)}^*(\omega(X_f,X_g))=\pi_{(H)}^*(\omega_{(H)}(Y_1,Y_2)).$
\end{lemma}
\begin{proof}
By \hypref{l:jpois1}{Lemma}, since $X_f\in\ham^G(M)$, there exists $Y_1\in\ham((Z/G)_{(H)})$ which is Hamiltonian with respect to a function $h_1\in\CIN((Z/G)_{(H)})$ satisfying $h_1=j^*((\pi^*)^{-1}f)|_{(Z/G)_{(H)}}$.  Similarly, there exists $Y_2\in\ham((Z/G)_{(H)})$ which is Hamiltonian with respect to $h_2=j^*((\pi^*)^{-1}f)|_{(Z/G)_{(H)}}$.  So,
\begin{align*}
i_{(H)}^*(\omega(X_f,X_g))=&~i_{(H)}^*\omega(X_f|_{Z_{(H)}},X_g|_{Z_{(H)}})&&\text{by \hypref{p:leftinvtham}{Proposition}}\\
=&~\pi_{(H)}^*\omega_{(H)}(X_f|_{Z_{(H)}},X_g|_{Z_{(H)}})\\
=&~\pi_{(H)}^*(\omega_{(H)}(Y_1,Y_2)).
\end{align*}
\end{proof}

\subsection*{\emph{Poisson and Symplectic Structures}}

\begin{definition}
A Poisson bracket on a differential space $(X,\mathcal{F})$ is a Lie bracket $\pois{}{}$ satisfying for any $f,g,h\in\mathcal{F}$: $$\pois{f}{gh}=h\pois{f}{g}+g\pois{f}{h}.$$
\end{definition}

\begin{example}\labell{x:stdpois}
Define $\pois{}{}$ on $(M,\omega)$ by $$\pois{f}{g}:=\omega(X_f,X_g).$$  This is the standard Poisson structure on a symplectic manifold.
\end{example}

Since the manifolds $(Z/G)_{(H)}$ are symplectic, their rings of functions admit Poisson structures $\pois{\cdot}{\cdot}_{(H)}$ as in \hypref{x:stdpois}{Example}.  In fact, we can define a Poisson bracket on all of $Z/G$ as follows.

\begin{definition}
Let $f,g\in\CIN(Z/G)$, and let $x\in(Z/G)_{(H)}$ for some $H\leq G$.  Then define $$\pois{f}{g}_{Z/G}(x):=\pois{f|_{(Z/G)_{(H)}}}{g|_{(Z/G)_{(H)}}}_{(H)}(x).$$
\end{definition}

\begin{proposition}[Lerman-Sjamaar]
The above bracket defines a Poisson bracket on $\CIN(Z/G)$.
\end{proposition}
\begin{proof}
See \cite{lerman-sjamaar}.
\end{proof}

We can also define a Poisson structure on $\CIN(M/G)$:

\begin{definition}
Let $f,g\in\CIN(M/G)$.  Then define $\pois{f}{g}_{M/G}:=(\pi^*)^{-1}\pois{\pi^*f}{\pi^*g}$, where $(\pi^*)^{-1}$ is the inverse of the isomorphism $\pi^*:\CIN(M/G)\to\CIN(M)^G$.
\end{definition}

\begin{proposition}
The above bracket defines a Poisson structure on $\CIN(M/G)$.
\end{proposition}
\begin{proof}
$\CIN(M)^G$ is a \emph{Poisson subalgebra} of $\CIN(M)$ and $\pi^*$ is an isomorphism between $\CIN(M/G)$ and $\CIN(M)^G$. The result follows.
\end{proof}

\begin{definition}
Let $(X,\mathcal{F},\pois{}{}_X)$ and $(Y,\mathcal{G},\pois{}{}_Y)$ be differential spaces equipped with Poisson structures.  A smooth map $F:X\to Y$ is \emph{Poisson} if for every $f,g\in\mathcal{G}$, $F^*(\pois{f}{g}_Y)=\pois{F^*f}{F^*g}_X$.
\end{definition}

\begin{proposition}
$\pi$ and $j$ are Poisson morphisms.
\end{proposition}
\begin{proof}
By definition of the Poisson structure on $\CIN(M/G)$, $\pi$ is a Poisson map.  As for $j$, fix $z\in Z_{(H)}$ and let $x=\pi_{(H)}(z)$. Let $f,g\in\CIN(M/G)$.  Then,
\begin{align*}
j^*\pois{f}{g}_{M/G}(x)=&~\pois{f}{g}_{M/G}(j(x))\\
=&~\pois{\pi^*f}{\pi^*g}(i(z))\\
=&~\omega(X_{\pi^*f},X_{\pi^*g})(i(z))\\
=&~i_{(H)}^*(\omega(X_{\pi^*f},X_{\pi^*g}))(z)\\
=&~\pi_{(H)}^*(\omega_{(H)}(Y_1,Y_2))(z)&&\text{by \hypref{l:jpois2}{Lemma}}
\end{align*}
where $Y_1=X_{h_1}$ and $Y_2=X_{h_2}$ with $h_1=j^*((\pi^*)^{-1}\pi^*f)|_{(Z/G)_{(H)}}=j^*f|_{(Z/G)_{(H)}}$, and a similar formula for $h_2$ replacing $f$ with $g$.  Thus, we have
$$
j^*\pois{f}{g}_{M/G}(x)=\pois{j^*f|_{(Z/G)_{(H)}}}{j^*g|_{(Z/G)_{(H)}}}_{(H)}(x)
=\pois{j^*f,j^*g}_{Z/G}(x).
$$
This completes the proof.
\end{proof}

Before continuing, we summarise our setting so far.  \hypref{d:main}{Diagram} is in the category of smooth stratified spaces; $M$, $M/G$, and $Z/G$ come equipped with Poisson structures on their rings of smooth functions, and $\pi$ and $j$ are Poisson morphisms with respect to these Poisson structures.
\section{The Zariski Tangent Bundle}\labell{s:zariski}

In this section, we review the basics of subcartesian space theory; in particular, those pertaining to the Zariski tangent space.  For more details, see \cite{LSW}.  Fix a subcartesian space $S$.

\begin{definition}[Zariski Tangent Bundle]
Given a point $x\in S$, a \emph{derivation} of $\CIN(S)$ at $x$ is a linear map $v:\CIN(S)\to\RR$ that satisfies Leibniz' rule: for all $f,g\in\CIN(S)$, $$v(fg)=f(x)v(g)+g(x)v(f).$$  The set of all derivations of $\CIN(S)$ at $x$ forms a vector space, called the \emph{(Zariski) tangent space} of $x$, and is denoted $T_xS$. Define the \emph{(Zariski) tangent bundle} $TS$ to be the (disjoint) union $$TS:=\bigcup_{x\in S}T_xS.$$  Denote the canonical projection $TS\to S$ by $\tau$.
\end{definition}

$TS$ is a subcartesian space with its differential structure generated by functions $f\circ\tau$ and $df$ where $f\in\CIN(S)$ and $d$ is the differential operator $df(v):=v(f)$.  The projection $\tau$ is smooth with respect to this differential structure.  Given a chart $\varphi:U\to\tilde{U}\subseteq\RR^n$ on $S$, $(\varphi\circ\tau,\varphi_*|_{\varphi\circ\tau}):TS\to T\RR^n\cong\RR^{2n}$ is a fibrewise linear chart on $TS$.  We will denote this chart by $\varphi_*$ henceforth.

\begin{definition}[Pushforward]
Let $R$ and $S$ be subcartesian spaces, and let $F:R\to S$ be a smooth map.  Then there is an induced fibrewise linear smooth map $F_*:TR\to TS$ defined by $$(F_*v)f=v(F^*f)$$ for all $v\in TR$ and $f\in\CIN(S)$.  $F_*$ satisfies the following commutative diagram.

$$\xymatrix{
TR \ar[r]^{F_*} \ar[d]_{\tau} & TS \ar[d]^{\tau} \\
R \ar[r]_{F} & S \\
}$$

$F_*$ is called the \emph{pushforward} of $F$, and is sometimes denoted as $dF$ or $TF$.
\end{definition}

We recall some notation.  For a subset $A\subseteq\RR^n$, let $\mathfrak{n}(A)$ be the ideal of the ring of smooth functions on $\RR^n$ consisting of functions that vanish on $A$.

\begin{proposition}[Local Representatives of Vectors]\labell{p:charTS}
Let $x\in S$ and let $\varphi:U\to\tilde{U}\subseteq\RR^n$ be a chart about $x$.  Then, $\tilde{v}\in T_{\varphi(x)}\RR^n$ is equal to $\varphi_*v$ for some $v\in T_xS$ if and only if $\tilde{v}(\mathfrak{n}(\varphi(U)))=\{0\}$.
\end{proposition}
\begin{proof}
See \cite{LSW}
\end{proof}

\begin{definition}[Derivations of $\CIN(S)$]
A \emph{(global) derivation} of $\CIN(S)$ is a linear map $X:\CIN(S)\to\CIN(S)$ that satisfies Leibniz' rule: for any $f,g\in\CIN(S)$, $$X(fg)=fX(g)+gX(f).$$  Denote the $\CIN(S)$-module of all derivations by $\Der\CIN(S)$.
\end{definition}

\begin{proposition}[$\Der\CIN(S)$ is a Lie Algebra]
The set of derivations of $\CIN(S)$ is a Lie algebra under the commutator bracket, and can be identified with the smooth sections of $\tau:TS\to S$.
\end{proposition}
\begin{proof}
See \cite{LSW}.
\end{proof}

\begin{proposition}[Local Representatives of Derivations]\labell{p:charDer}
Let $x\in S$, and let $\varphi:U\to\tilde{U}\subseteq\RR^n$ be a chart about $x$.  Let $\tilde{X}\in\Der\CIN(\RR^n)$.  Then $\tilde{X}$ satisfies $$\varphi_*(X|_U)=\tilde{X}|_{\tilde{U}}$$ for some derivation $X\in\Der\CIN(S)$ if and only if $$\tilde{X}(\mathfrak{n}(\tilde{U}))\subseteq\mathfrak{n}(\tilde{U}).$$ Moreover, for any $X\in\Der\CIN(S)$, there exist an open neighbourhood $V\subseteq U$ of $x$ and $\tilde{X}\in\Der\CIN(\RR^n)$ satisfying $\varphi_*(X|_V)=\tilde{X}|_{\varphi(V)}$. We call $\tilde{X}$ a \emph{local extension} or a \emph{local representative} of $X$ with respect to $\varphi$.
\end{proposition}
\begin{proof}
See \cite{LSW}.
\end{proof}

\begin{definition}[Locally Trivial Surjections]
Let $R$ and $S$ be subcartesian spaces, and let $f$ be a surjective smooth map between them.  Then $f:R\to S$ is \emph{locally trivial} if for every $x\in S$ there exist an open neighbourhood $U\subseteq S$ of $x$, a subcartesian space $F$, and a diffeomorphism $\psi:f^{-1}(U)\to U\times F$ such that the following diagram commutes ($\operatorname{pr}_1$ being the projection of the first component.)

$$\xymatrix{
f^{-1}(U) \ar[rr]^{\psi} \ar[dr]_f & & U\times F \ar[dl]^{\operatorname{pr}_1}\\
& U & \\
}$$
\end{definition}

\begin{theorem}[Local Triviality of $TS$]\labell{t:TSloctriv}
Let $S$ be a subcartesian space.  There exists an open dense subset $U\subseteq S$ such that $\tau|_U:TS|_U\to U$ is locally trivial.
\end{theorem}
\begin{proof}
See \cite{LSW}.
\end{proof}

\begin{corollary}
The $k$th exterior product of fibres of $TS$ over $S$ are also locally trivial over an open dense subset of $S$.
\end{corollary} 
\section{Vector Fields on Subcartesian Spaces}\labell{s:vectorfields}

In this section we review the theory of vector fields on subcartesian spaces, developed by \'Sniatycki in \cite{sniatycki}.

\begin{definition}[Integral Curves]
Fix $X\in\Der\CIN(S)$ and $x\in S$.  A \emph{maximal integral curve} $\exp(\cdot X)(x)$ of $X$ through $x$ is a smooth map from a connected subset $I^X_x\subseteq\RR$ containing 0 to $S$ such that $\exp(0X)(x)=x$, the following diagram commutes,
$$\xymatrix{
TI^X_x \ar[rr]_{\exp(\cdot X)(x)_*} && TS\\
I^X_x \ar[u]^{\frac{d}{dt}} \ar[rr]^{\exp(\cdot X)(x)} && S \ar[u]_{X}
}$$
and such that $I_x^X$ is maximal among the domains of all such curves.  In particular, for all $f\in\CIN(S)$ and $t\in I^X_x$, $$\frac{d}{dt}(f\circ\exp(tX)(x))=(Xf)(\exp(tX)(x)).$$ We adopt the convention that the map $c:\{0\}\to S:0\mapsto c(0)$ is an integral curve of every global derivation of $\CIN(S)$.
\end{definition}

\begin{theorem}[ODE Theorem for Subcartesian Spaces -- \'Sniatycki]\labell{t:ode}
Let $S$ be a subcartesian space, and let $X\in\Der\CIN(S)$.  Then, for any $x\in S$, there exists a unique maximal integral curve $\exp(\cdot X)(x)$ through $x$.
\end{theorem}
\begin{proof}
See \cite{sniatycki0} and \S4 Theorem 1 of \cite{sniatycki}.
\end{proof}

\begin{proposition}[Local Representatives of Integral Curves]\labell{p:intcurve}
Let $\varphi:U\to\tilde{U}\subseteq\RR^n$ be a chart on $S$, $X\in\Der\CIN(S)$ and $\tilde{X}\in\Der\CIN(\RR^n)$ such that $$\varphi_*(X|_U)=\tilde{X}|_{\tilde{U}}.$$  Then for all $x\in S$ and $t\in I^{X|_U}_x$,  $$\varphi(\exp(tX)(x))=\exp(t\tilde{X})(\varphi(x)).$$
\end{proposition}

\begin{proof}
Denote by $J$ the open subset of $I^X_x$ such that for every $t\in J$, $\exp(tX)(x)\in U$. Define $\gamma:J\to\tilde{U}:t\mapsto\varphi(\exp(tX)(x))$.  Then,
\begin{align*}
\frac{d}{dt}\Big|_{t=0}(\gamma(t))=&~\varphi_*(X|_x)\\
=&~\tilde{X}|_{\varphi(x)}.
\end{align*}
Applying the ODE theorem, $\gamma(t)=\exp(t\tilde{X})(\varphi(x))$.
\end{proof}

Fix a derivation $X\in\Der\CIN(S)$.  Let $A^X\subseteq\RR\times S$ be defined as $$A^X:=\coprod_{x\in S}I_x^X.$$  Then there is an induced smooth map $A^X\to S$ whose restriction to each fibre $A^X\cap(\RR\times\{x\})$ is the domain $I^X_x$ of the maximal integral curve $\exp(\cdot X)(x)$.  

\begin{definition}[Local Flows]
Let $D$ be a subset of $\RR\times S$ containing $\{0\}\times S$ such that $D\cap(\RR\times\{x\})$ is connected for each $x\in S$.  A map $\phi:D\to S$ is a \emph{local flow} if $D$ is open, $\phi(0,x)=x$ for each $x\in S$, and $\phi(t,\phi(s,x))=\phi(t+s,x)$ for all $x\in S$ and $s,t\in\RR$ for which both sides are defined.
\end{definition}

\begin{remark}
If $S$ is a smooth manifold, then every derivation $X$ admits a local flow $\exp(\cdot X)(\cdot)$ sending $(t,x)$ to $\exp(tX)(x)$.  This is not the case with subcartesian spaces.  Indeed, consider the closed ray $[0,\infty)$, and the global derivation $X=\partial_x$.  Then the domain $D$ of $\exp(\cdot X)(\cdot)$ is not an open subset of $\RR\times[0,\infty)$.  Indeed, $D\cap(\RR\times\{x\})=[-x,\infty)\times\{x\}$ for each $x\in\RR$.  Thus, $D=\{(t,x)\in\RR^2~|~t\geq-x,~x\geq0\}$.  This motivates the following definition.
\end{remark}

\begin{definition}[Vector Fields]
A \emph{vector field} on $S$ is a derivation $X$ of $\CIN(S)$ such that $A^X$ is open in $\RR\times S$.  Equivalently, the map $(t,x)\mapsto\exp(tX)(x)$ defined on $A^X$ is a local flow.  Here let us emphasise that $\exp(tX)(x)$ is the maximal integral curve through $x$.  Denote the set of all vector fields on $S$ by $\Vect(S)$.
\end{definition}

\begin{remark}
Given a vector field $X$ on $S$, since $A^X$ is open, the domain of each of its maximal integral curves is open.  Note, however, that the converse is not true: if $X$ is a global derivation and each of its maximal integral curves has an open domain, it is not necessarily true that $X$ is a vector field. For a counterexample, see \hypref{x:loccmpctneeded}{Example}.
\end{remark}

For the important proposition to come, we recall the concepts of ``locally closed'' and ``locally compact''.  In the literature (for example, \cite{sniatycki}), the notion of locally closed is used for subsets of $\RR^n$ (in particular, for differential subspaces of $\RR^n$).  ``Locally compact'', however, can be used for subcartesian spaces (or any topological space), not just differential subspaces of $\RR^n$.  It also tends to be more widely used in the literature.  We show in the following lemma that, for differential subspaces of $\RR^n$, these two concepts coincide. Before stating and proving the lemma, we recall the definitions of locally compact and locally closed subsets.

\begin{itemize}
\item Let $S\subseteq\RR^n$.  $S$ is \emph{locally compact} if for every $x\in S$ there exist a relatively open neighbourhood $U\subseteq S$ of $x$ and a compact set $K\subseteq S$ such that $U\subseteq K$.

\item Let $S\subseteq\RR^n$.  $S$ is \emph{locally closed} if for every $x\in S$ there exist an open neighbourhood $V\subseteq\RR^n$ of $x$ and a closed set $C\subseteq\RR^n$ such that $V\cap C$ is a relatively open neighbourhood of $x$ in $S$.
\end{itemize}

\begin{lemma}
Let $S\subseteq\RR^n$.  Then $S$ is locally closed if and only if $S$ is locally compact.
\end{lemma}

\begin{proof}
Assume that $S$ is locally compact, and fix $x\in S$.  Then, there exist an open neighbourhood $U\subseteq S$ of $x$ and a compact $K\subseteq S$ such that $U\subseteq K$.  There exists an open neighbourhood $V\subseteq\RR^n$ of $x$ such that $U=V\cap S$, and $K$ is a compact subset of $\RR^n$ and hence closed.  $$V\cap K\subseteq V\cap S=U\subseteq V\cap K$$
and so $V\cap K=U$.  Hence, $S$ is locally closed.\\

Conversely, assume $S$ is locally closed, and fix $x\in S$.  There exist an open neighbourhood $V\subseteq\RR^n$ of $x$ and a closed subset $C\subseteq\RR^n$ such that $V\cap C$ is an open neighbourhood of $x$ contained in $S$.  Let $B\subseteq\RR^n$ be the open ball of radius $\epsilon>0$ centred at $x$.  Then, $B\cap V\cap C$ is an open neighbourhood of $x$ in $S$.  Choosing $\epsilon$ to be sufficiently small so that $\overline{B}\subseteq V$, we have $B\cap V\cap C=B\cap C$ and $\overline{B}\cap C\subseteq S$.  Since $\overline{B}$ and $C$ are closed subsets of $\RR^n$, their intersection is closed.  Since this intersection is contained in $S$, $\overline{B}\cap C$ is a closed subset of $S$.  Moreover, since $\overline{B}$ is compact in $\RR^n$, $\overline{B}\cap C$ is compact in $\RR^n$ as well.\\

Now, let $\{W_\alpha\}_{\alpha\in A}$ be an open cover of $\overline{B}\cap C$ in $S$. Then, for each $\alpha$, there exists an open set $\tilde{W}_\alpha$ such that $W_\alpha=\tilde{W}_\alpha\cap S$. Thus, the collection of open sets $\{\tilde{W}_\alpha\}_{\alpha\in A}$ forms an open cover of $\overline{B}\cap C$ in $\RR^n$.  Since $\overline{B}\cap C$ is compact in $\RR^n$, there is a finite subcover $\{\tilde{W}_{\alpha_i}\}$ of $\{\tilde{W}_\alpha\}$ that covers $\overline{B}\cap C$.  But then for each $\alpha_i$, $W_{\alpha_i}:=\tilde{W}_{\alpha_i}\cap S$ is an open subset of $S$, contained in $\{W_\alpha\}$.  We conclude that the collection $\{W_{\alpha_i}\}$ is a finite subcover of $\{W_\alpha\}$ covering $\overline{B}\cap C$, and hence $\overline{B}\cap C$ is compact as a subset of $S$.
\end{proof}

Note that a subcartesian space can be locally compact, which extends the notion of local closedness beyond differential subspaces of $\RR^n$.

\begin{proposition}[Integral Curve Domains -- \'Sniatycki]\labell{p:opendomain}
Let $S$ be a locally compact subcartesian space.  A derivation $X$ of $\CIN(S)$ is a vector field if and only if the domain of each of its maximal integral curves is open.
\end{proposition}
\begin{proof}
See \S4 Proposition 3 of \cite{sniatycki}.
\end{proof}

\begin{example}[\'Sniatycki \cite{sniatycki}]\labell{x:loccmpctneeded}
Let $S$ be the differential subspace of $\RR^2$ given by $$S=\{(x,y)~|~x^2+(y-1)^2<1\}\cup\{(x,y)~|~y=0\}.$$  Consider the global derivation given by the restriction of $\partial_x$ to $S$.  Then the domain of each maximal integral curve of $\partial_x$ is open; however, at the origin, the integral curve does not induce a local diffeomorphism.
\end{example} 
\section{Locally Complete Families of Vector Fields and Smooth Stratified Spaces}\labell{s:locallycomplete}

We are interested in using families of vector fields in order to obtain a ``nice'' partition of a subcartesian space.  The condition needed to achieve this on these families is defined next.  We then give examples.

\begin{definition}[Locally Complete Families of Vector Fields]
A family of vector fields $\mathcal{F}\subseteq\Vect(S)$ is \emph{locally complete} if for every $X,Y\in\mathcal{F}$, every $x\in S$ and every $t\in\RR$ such that $(\exp(tX)_*Y)|_x$ is well-defined, there exist an open neighbourhood $U$ of $x$ and a vector field $Z\in\mathcal{F}$ such that $\exp(tX)_*Y|_U=Z|_U$.
\end{definition}

\begin{remark}
Note that for $f\in\CIN(S)$ and $X,Y\in\mathcal{F}$ where $\mathcal{F}$ is a locally complete family of vector fields, $x\in S$ and $s,t\in\RR$, we have (where it is defined)
\begin{align*}
\frac{d}{ds}f(\exp(tX)(\exp(sY)(x)))=&~(\exp(tX)_*(Y|_{\exp(sY)(x)}))f\\
=&~((\exp(tX)_*Y)f)(\exp(tX)(\exp(sY)(x))).
\end{align*}
For fixed $t$, there exists an open neighbourhood $U$ of $x$ on which the local flow of $(\exp(tX)_*Y)|_U$ is equal to $s\mapsto\exp(tX)(\exp(sY)(y))$ for $y\in U$.
\end{remark}

\begin{example}[Not Locally Complete]\labell{x:lcexample}
Consider $S=\RR^2$, and let $\mathcal{F}$ be the family of all $\RR$-linear combinations of the two vector fields $\partial_x$ and $x\partial_y$.  This family is not locally complete, as one can check that $\exp(tx\partial_y)_*\partial_x=\partial_x+t\partial_y$ is not contained in $\mathcal{F}$ for any $t\neq 0$.
\end{example}

\begin{proposition}[\'Sniatycki]\labell{p:vectloccompl}
$\Vect(S)$ is locally complete.
\end{proposition}
\begin{proof}
See \S4 Theorem 2 of \cite{sniatycki}.
\end{proof}

We will later give examples of subcartesian spaces equipped with smooth stratifications on which we apply the theory of vector fields, but in order to do this we will need some more terminology coming from the theory of stratified and decomposed spaces.
\begin{definition}[Refinements of Decomposed Spaces]
Fix a differential space $X$ with smooth decompositions $\mathcal{D}_1$ and $\mathcal{D}_2$.  $\mathcal{D}_1$ is a \emph{refinement} of $\mathcal{D}_2$, denoted $\mathcal{D}_1\geq\mathcal{D}_2$ if for every piece $P_1\in\mathcal{D}_1$, there exists $P_2\in\mathcal{D}_2$ such that $P_1\subseteq P_2$.  This induces a partial ordering on the set of decompositions on $X$. We say that $\mathcal{D}$ is \emph{minimal} if for any $\mathcal{D}'$ such that $\mathcal{D}\geq\mathcal{D}'$, we have $\mathcal{D}=\mathcal{D}'$.
\end{definition}

\begin{example}
The square $[0,1]^2$ with the decomposition given in \hypref{x:square}{Example} is minimal.
\end{example}

\begin{theorem}[Bierstone]\labell{t:bierstone}
If $G$ is a compact Lie group acting on a manifold $M$, then the orbit-type stratification of $M/G$ is minimal.
\end{theorem}
\begin{proof}
See \cite{bierstone1} and \cite{bierstone2}.
\end{proof}

\begin{definition}[Stratified Vector Fields]
Let $S$ be a smooth stratified space.  Let $X\in\Vect(S)$.  If for each stratum $P$ of $S$, $X|_P$ is a smooth vector field on $P$ as a smooth manifold, then we call $X$ \emph{stratified}. Denote the set of all stratified vector fields on $S$ by $\Vect_{strat}(S)$.
\end{definition}

\begin{remark}
As noted before, different terminology and definitions appear in the literature.  For example, in \cite{sniatycki}, \'Sniatycki defines a stratified vector field on a smooth stratified space $S$ as a continuous section (not necessarily smooth) of the Zariski tangent bundle that is smooth on the strata.  A strongly stratified vector field is an element of $\Vect(S)$ that restricts to a smooth section of each strata, which is the definition that we have taken for a stratified vector field.
\end{remark}

\begin{proposition}[\'Sniatycki]\labell{p:stratloccompl}
Let $S$ be a smooth stratified space.  Then $\Vect_{strat}(S)$ forms a locally complete family.
\end{proposition}
\begin{proof}
See \S6 Lemma 11 of \cite{sniatycki}
\end{proof}

Now, let $G$ be a compact Lie group acting on a manifold $M$.

\begin{proposition}\labell{p:vectGloccompl}
$\Vect(M)^G$ is a locally complete Lie subalgebra of $\Vect(M)$.
\end{proposition}

\begin{proof}
For any two invariant vector fields $X$ and $Y$, we have for all $g\in G$, $$g_*[X,Y]=[g_*X,g_*Y]=[X,Y],$$ and for $x\in M$,  $$g\cdot\exp(tX)(\exp(sY)(x))=\exp(tX)(\exp(sY)(g\cdot x))$$ for $s,t$ such that the composition of the curves is defined. Thus $\exp(tX)_*Y$ is locally defined about $G$-orbits.  Since $\Vect(M)$ is locally complete, for any $x\in M$ there exist a vector field $Z$ on $M$ and an open neighbourhood $U$ of $x$ such that $\exp(tX)_*Y$ is defined on $U$ and $(\exp(tX)_*Y)|_U=Z|_U$.  Since $\exp(tX)_*Y$ is invariant about $x$, we can choose $U$ to be a $G$-invariant open neighbourhood.  Let $V\subset U$ be a $G$-invariant open neighbourhood of $x$ such that $\overline{V}\subset U$.  Let $b:M\to\RR$ be a $G$-invariant smooth bump function with support in $U$ and $b|_V=1$.  Then, $bZ\in\Vect(M)^G$ extends $(\exp(tX)_*Y)|_V$ to a invariant vector field on $M$.
\end{proof}

\begin{definition}
Identify $\g$ with the invariant (under left multiplication) vector fields on $G$. Let $\rho:\g\to\Vect(M)$ be the $\g$-representation induced by the $G$-action.
\end{definition}

\begin{proposition}\labell{p:rhogloccompl}
$\rho(\g)$ is a locally complete Lie subalgebra of $\Vect(M)$.
\end{proposition}

\begin{proof}
Let $\xi,\zeta\in\g$, and let $\xi_M=\rho(\xi)$ and $\zeta_M=\rho(\zeta)$.  Then, $\exp(t\xi_M)_*\zeta_M=(\Ad_{\exp(t\xi)}\zeta)_M$.
\end{proof}

Recall that for a compact Lie group $G$, its Lie algebra decomposes as a direct sum of the derived Lie subalgebra and the centre of $\g$: $$\g=[\g,\g]\oplus\mathfrak{z}(\g).$$

\begin{corollary}\labell{c:rhogloccompl}
$\rho([\g,\g])$ and $\rho(\mathfrak{z}(\g))$ are locally complete Lie subalgebras of $\Vect(M)$.
\end{corollary}

\begin{proof}
Since $[\g,\g]$ and $\mathfrak{z}(\g)$ are themselves Lie algebras corresponding to compact Lie groups, this corollary is immediate from the above lemma.
\end{proof}

\begin{definition}\labell{d:A}
Define $\mathcal{A}$ to be the smallest Lie subalgebra of $\Vect(M)$ containing $\Vect(M)^G$ and $\rho(\g)$.
\end{definition}

\begin{remark}
Note that $\mathcal{A}$, $\Vect(M)^G$ and $\rho(\g)$ are not necessarily closed under multiplication by functions in $\CIN(M)$, but $\Vect(M)^G$ is closed under multiplication by $G$-invariant smooth functions.
\end{remark}

\begin{proposition}\labell{p:A}
$\mathcal{A}$ is locally complete, and it is a direct sum of Lie algebras: $$\mathcal{A}=\rho([\g,\g])\oplus\Vect(M)^G.$$
\end{proposition}

\begin{proof}
Let $\xi\in\g$ and $X\in\Vect(M)^G$.  Then, $$[\xi_M,X]=\underset{t\to0}{\lim}\frac{\exp(t\xi_M)_*(X|_{\exp(-t\xi_M)(x)})-X|_x}{t}=0$$ since $\exp(t\xi_M)_*(X|_{\exp(-t\xi_M)(x)})=X|_x$ by left-invariance.  Thus,
\begin{equation}\labell{e:flowscommute}
\exp(t\xi_M)\circ\exp(sX)=\exp(sX)\circ\exp(t\xi_M).
\end{equation}
Now, let $\xi\in\g$ and assume for all $g\in G$ and $x\in M$, we have $$g_*(\xi_M|_x)=\xi_M|_{g\cdot x};$$ that is, $\xi_M$ is invariant.  Then, $$\frac{d}{dt}\Big|_{t=0}(g\cdot\exp(t\xi_M)(x))=\frac{d}{dt}\Big|_{t=0}\exp(t\xi_M)(g\cdot x).$$  The uniqueness property of $\exp$ implies that $$g\cdot\exp(t\xi_M)(x)=\exp(t\xi_M)(g\cdot x).$$  Hence  $(g\exp(t\xi))\cdot x=(\exp(t\xi)g)\cdot x$.  Since this is true for all $g\in G$, $\exp(t\xi)$ must be in the centre of $G$, and hence $\xi\in\mathfrak{z}(\g)$.  Thus, $$\rho(\g)\cap\Vect(M)^G=\rho(\mathfrak{z}(\g)).$$  Since $\rho$ is a Lie algebra homomorphism, from \hypref{e:flowscommute}{Equation}: $\rho(\g)=\rho([\g,\g])\oplus\rho(\mathfrak{z}(\g))$, and we obtain the direct sum structure of $\mathcal{A}$.\\
To show local completeness, by \hypref{p:vectGloccompl}{Proposition} and \hypref{p:rhogloccompl}{Proposition} it suffices to show that for any $\xi\in\g$ and $X\in\Vect(M)^G$, $\exp(t\xi_M)_*X\in\mathcal{A}$ and $\exp(tX)_*\xi_M\in\mathcal{A}$.  The former is immediate since $X$ is invariant.  The latter follows from \hypref{e:flowscommute}{Equation}:
\begin{align*}
\exp(tX)_*(\xi_M|_x)=&~\frac{d}{ds}\Big|_{s=0}\exp(tX)(\exp(s\xi_M)(x))\\
=&~\frac{d}{ds}\Big|_{s=0}\exp(s\xi_M)(\exp(tX)(x))\\
=&~\xi_M|_{\exp(tX)(x)}.
\end{align*}
\end{proof}

We return to Hamiltonian group action setting.  Recall the Poisson structure on $Z/G$, and that we can defined Hamiltonian vector fields on $Z/G$.

$$\xymatrix{
Z \ar[d]_{\pi_Z} \ar[r]^i & M \ar[d]^{\pi} \\
Z/G \ar[r]_j & M/G \\
}$$

\begin{lemma}
For any $h\in\CIN(Z/G)$, the derivation $\{h,\cdot\}_{Z/G}$ is a Hamiltonian vector field.
\end{lemma}

\begin{proof}
Sjamaar-Lerman prove the existence and uniqueness of maximal integral curves of these derivations, and that they remain in the orbit-type strata (see \cite{lerman-sjamaar}).  Since these strata are manifolds, the maximal integral curves have open domains.  Hence, by \hypref{p:opendomain}{Proposition}, they are vector fields.
\end{proof}

\begin{proposition}
$\ham(Z/G)$ is a locally complete family in $\Vect(Z/G)$.
\end{proposition}
\begin{proof}
See \S7 Proposition 4 in \cite{sniatycki}.
\end{proof} 
\section{The Orbital Tangent Bundle}\labell{s:orbitaltangentbundle}

In this section, for a fixed family of vector fields we introduce a ``subbundle'' of the Zariski tangent bundle consisting of vectors that are fibrewise linear combinations of vectors in the images of vector fields in the family.  We show that the family of \emph{all} vector fields yields such a tangent bundle that is locally trivial on an open dense subset, as well as equal to the Zariski tangent bundle over an open dense subset.

\begin{definition}[Orbital Tangent Bundle]
Let $\mathcal{F}$ be a family of vector fields on $S$.  For each $x\in S$, denote by $\widehat{T}^\mathcal{F}_xS$ the linear subspace of $T_xS$ spanned by all vectors $v\in T_xS$ such that there exists a vector field $X\in\mathcal{F}$ with $v=X|_x$.  If $\mathcal{F}=\Vect(S)$, then we will denote this space by $\widehat{T}_xS$.  We will call $\widehat{T}^\mathcal{F}_xS$ the \emph{orbital tangent space} of $S$ at $x$ with respect to $\mathcal{F}$.  Let $\widehat{T}^\mathcal{F}S$ be the (disjoint) union $$\widehat{T}^\mathcal{F}S:=\bigcup_{x\in S}\widehat{T}^\mathcal{F}_xS.$$  We will call $\widehat{T}^\mathcal{F}S$ the \emph{orbital tangent bundle} with respect to $\mathcal{F}$.  It is a differential subspace of $TS$.  Denote by $\widehat{\tau}_\mathcal{F}$ the restriction of $\tau:TS\to S$ to $\widehat{T}^\mathcal{F}S$ and by $\delta_\mathcal{F}(x)$ the dimension $\dim(\widehat{T}^\mathcal{F}_xS)$.
\end{definition}

\begin{remark}
Since $\widehat{T}^\mathcal{F}S$ is a differential subspace of $TS$, a chart $\varphi:U\to\tilde{U}\subseteq\RR^n$ on $S$ induces a chart $(\varphi\circ\widehat{\tau}_{\mathcal{F}},\varphi_*|_{\varphi\circ\widehat{\tau}_{\mathcal{F}}})$ on $\widehat{T}^\mathcal{F}S$, which we shall denote simply as $\varphi_*$.  This is just a restriction of the corresponding chart on $TS$.  It makes the following diagram commute.

$$\xymatrix{
\widehat{T}^\mathcal{F}S|_U \ar[d]_{\widehat{\tau}_{\mathcal{F}}} \ar[r]^{\varphi_*} & T\RR^n \ar[d]^{\tau} \\
U \ar[r]_{\varphi} & \RR^n
}$$

This extends to (fibred) exterior powers of $\widehat{T}^\mathcal{F}S$ in the natural way; \emph{i.e.} to $$\bigwedge_S^k\widehat{T}^\mathcal{F}S:=\bigcup_{x\in S}\bigwedge^k\widehat{T}^\mathcal{F}_xS.$$
\end{remark}

\begin{lemma}\labell{l:lwrsemicont}
The map $\delta_\mathcal{F}:S\to\ZZ$ is lower semicontinuous.
\end{lemma}

\begin{proof}
Define $S_i:=\{x\in S~|~\delta_{\mathcal{F}}(x)\geq i\}$.  The goal is to show that $S_i$ is open for each $i$.  Let $y\in S_i$.  Then there exist $Y_1,...,Y_k\in\mathcal{F}$, where $k\geq i$, such that $\{Y_1|_y,...,Y_k|_y\}$ is a basis for $\widehat{T}^\mathcal{F}_yS$.  Linear independence is an open condition, and so there exists an open neighbourhood $U$ of $y$ such that $\{Y_1|_z,...,Y_k|_z\}$ is linear independent for all $z\in U$.  Hence, $\widehat{T}^\mathcal{F}_zS$ contains the span of $\{Y_1|_z,...,Y_k|_z\}$ as a linear subspace for each $z\in U$.  Thus, $\delta_\mathcal{F}(z)\geq k\geq i$.  Thus, $U\subseteq S_i$.
\end{proof}

\begin{proposition}[Local Triviality of $\widehat{T}^\mathcal{F}S$]\labell{p:checkTSloctriv}
There exists an open dense subset $U\subseteq S$ such that $\widehat{\tau}_\mathcal{F}|_{U}:\widehat{T}^\mathcal{F}S|_U\to U$ is locally trivial.
\end{proposition}

\begin{proof}
We will show that for any point $x\in S$ and any open set $U$ containing $x$, there is a point $z\in U$ and an open neighbourhood $V\subseteq U$ of $z$ so that $\widehat{\tau}_\mathcal{F}^{-1}(V)\cong V\times F$ for some vector space $F$.\\

Fix $x\in S$.  Define $S_i$ as in the proof of \hypref{l:lwrsemicont}{Lemma}.  Define $$m:=\underset{V\backepsilon x}{\inf}\{\sup\{k~|~S_k\cap V\neq\emptyset\}\}$$ where $V$ runs through all open neighbourhoods of $x$.  There exists an open neighbourhood $W$ of $x$ such that $\sup_{z\in W}\{\delta_\mathcal{F}(z)\}=m$.  Now fix $z\in W$ such that $\delta_\mathcal{F}(z)=m$.  Then, there are vector fields $Y_1,...,Y_m\in\mathcal{F}$ such that $\{Y_1|_z,...,Y_m|_z\}$ spans $\widehat{T}^\mathcal{F}_zS$.  Since linear independence is an open condition and $m$ is maximal, there is an open neighbourhood $V\subseteq W$ of $z$ such that $\{Y_1|_y,...,Y_m|_y\}$ spans $\widehat{T}^\mathcal{F}_yS$ for all $y\in V$.  Hence, $\widehat{T}^\mathcal{F}S$ is locally trivial over $V$.\\

Now, let $U$ be any open subset containing $x$.  We claim that there exists some $z\in W\cap U$ such that $\delta_\mathcal{F}(z)=m$.  Assume otherwise.  If $\sup_{z\in W\cap U}(\delta_\mathcal{F}(z))>m$, then this contradicts the definition of $W$.  If $\sup_{z\in W\cap U}\{\delta_\mathcal{F}(z)\}<m$, then this contradicts the definition of $m$.  Now, choose an open neighbourhood $V\subseteq W\cap U$ of $z$ as above, and the result follows.
\end{proof}

\begin{corollary}\labell{c:checkTopendense}
Let $\mathcal{F}$ be a locally complete family of vector fields, and let $U\subseteq S$ be an open dense subset on which $\widehat{T}^\mathcal{F}S$ is locally trivial.  Then, $\widehat{\tau}_\mathcal{F}^{-1}(U)$ is open and dense in $\widehat{T}^\mathcal{F}S$.
\end{corollary}

\begin{proof}
By continuity, $\widehat{\tau}_\mathcal{F}(U)$ is open.  Let $x\in S\smallsetminus U$, and let $Y_1,...,Y_k\in\mathcal{F}$ such that $\{Y_1|_x,...,Y_k|_x\}$ forms a basis of $\widehat{T}^\mathcal{F}_xS$.  Since linear independence is an open condition, there is an open neighbourhood $V$ of $x$ on which $\{Y_1|_y,...,Y_k|_y\}$ is linear independent for all $y\in V$, and their span is a subset of $\widehat{T}^\mathcal{F}_yS$.  Hence, $\widehat{T}^\mathcal{F}_xS\subseteq\overline{\widehat{\tau}_\mathcal{F}^{-1}(U)}$.
\end{proof}

\begin{remark}
The above corollary extends to exterior powers of the fibres of  $\widehat{T}^{\mathcal{F}}S$; that is, there exists an open dense subset $U\subseteq S$ on which $\bigwedge_S^k\widehat{T}^\mathcal{F}S\Big|_U\to U$ is locally trivial.
\end{remark}

\begin{proposition}[Zariski Versus Orbital Tangent Bundles]\labell{p:checkZar}
Let $S$ be a locally compact subcartesian space. Then there exists an open dense subset $U\subseteq S$ such that for each $x\in U$, $$\widehat{T}_xS=T_xS.$$
\end{proposition}

\begin{proof}
By \hypref{t:TSloctriv}{Theorem} and \hypref{p:checkTSloctriv}{Proposition}, there exists an open dense subset $U\subseteq S$ on which $TS$ and $\widehat{T}S$ are locally trivial.  Let $x\in U$, and let $\varphi:V\to\tilde{V}\subseteq\RR^n$ be a chart about $x$ where $V\subseteq U$ and $n=\dim(T_xS)$ (see \cite{LSW}).  Then the derivations $\partial_1,...,\partial_n$ on $V$ arising from coordinates on $\RR^n$ give a local trivialisation of $TV$ (again, see \cite{LSW}).  Let $W_1$ and $W_2$ be open neighbourhoods of $x$ satisfying $\overline{W_1}\subset W_2\subset\overline{W_2}\subset V$.  Let $b:S\to\RR$ be a smooth bump function that is equal to 1 on $W_1$ and 0 outside of $W_2$.  Then $b\partial_1,...,b\partial_n$ extend to derivations on all of $S$, and we claim that they are vector fields.\\

Now, for $i=1,...,n$, shrinking $V$ if necessary, there exist $\tilde{X}_1,...,\tilde{X}_n\in\Der\CIN(\RR^n)$ satisfying $\varphi_*(b\partial_i)=\tilde{X}_i|_{\tilde{V}}$.  Each $\tilde{X}_i$ gives rise to a local flow $\exp(\cdot \tilde{X}_i)(\cdot)$, such that for each $y\in \tilde{V}$, $\exp(\cdot \tilde{X}_i)(\varphi(y))$ has an open domain.  By \hypref{p:intcurve}{Proposition}, $\exp(\cdot \tilde{X}_i)(\varphi(y))=\varphi(\exp(t b\partial_i)(y))$ for all $t\in I^{b\partial_i}_y$ for which the integral curve lies in $V$.  But since $b$ is supported in $V$, the entire curve $\exp(\cdot b\partial_i)(y)$ is in $V$.  Hence, $\exp(t \tilde{X}_i)(\varphi(y))\in\tilde{V}$ for all $t\in I^{\tilde{X}_i}_{\varphi(y)}$.  Since $\tilde{X}_i$ is a vector field on $\RR^n$, $I^{\tilde{X}_i}_{\varphi(y)}$ is open, and consequently so is $I^{b\partial_i}_y$.  Thus, by \hypref{p:opendomain}{Proposition} $b\partial_i$ is a vector field on $V$, and since it has been extended as 0 to the rest of $S$, it is a vector field on $S$.  Finally, since $(b\partial_i)|_{W_1}=\partial_i|_{W_1}$ for each $i$, we see that $\widehat{T}_yS=T_yS$ for all $y\in W_1$, since $T_yS$ is the span over $\RR$ of $\{\partial_1|_y,...,\partial_n|_y\}.$
\end{proof} 
\section{Orbits of Families of Vector Fields}\labell{s:orbits}

In this section we review the theory of orbits of families of vector fields, including the Orbit Theorem for subcartesian spaces, proven by \'Sniatycki in \cite{sniatycki}.

\begin{definition}[Orbits]
Let $S$ be a subcartesian space, and let $\mathcal{F}$ be a family of vector fields.  The \emph{orbit} of $\mathcal{F}$ through a point $x$, denoted $O^{\mathcal{F}}_x$ or just $O_x$ if $\mathcal{F}=\Vect(S)$, is the set of all points $y\in S$ such that there exist vector fields $X_1,...,X_k\in\mathcal{F}$ and real numbers $t_1,...,t_k\in\RR$ satisfying $$y=\exp(t_1X_1)\circ...\circ\exp(t_kX_k)(x).$$
Denote by $\mathcal{O}_{\mathcal{F}}$, or just $\mathcal{O}$ if $\mathcal{F}=\Vect(S)$, the set of all orbits $\{O^\mathcal{F}_x~|~x\in S\}$.  Note that $\mathcal{O}_{\mathcal{F}}$ induces a partition of $S$ into connected differential subspaces.
\end{definition}

Given a family of vector fields $\mathcal{F}$ on $S$, there exists a natural topology on the orbits that in general is finer than the subspace topology.  We define this topology here using similar notation as found in \cite{sniatycki} and \cite{sussmann}.  Let $X_1,...,X_k\in\mathcal{F}$.  Let $\xi:=(X_1,...,X_k)$ and $T=(t_1,...,t_k)$, and define $\xi_T(x):=\exp(t_kX_k)\circ...\circ\exp(t_1X_1)(x).$  $\xi_T(x)$ is well-defined for all $(T,x)$ in an open neighbourhood $U(\xi)$ of $(0,x)\in\RR^k\times S$.  Define $U_x(\xi)$ to be the set of all $T\in\RR^k$ such that $\xi_T(x)$ is well-defined; that is, $U_x(\xi)=U(\xi)\cap(\RR^k\times\{x\})$.  Let $i:O^\mathcal{F}_x\hookrightarrow S$ be the inclusion map.  Fix $y\in i(O^\mathcal{F}_x)$ and let $\varphi:V\to\tilde{V}\subseteq\RR^n$ be a chart of $S$ about $y$.  We give $W:=i^{-1}(V\cap i(O^\mathcal{F}_x))$ the strongest topology such that for each $\xi$ and $y\in i(W)$ the map $$\rho_{\xi,y}:U_y(\xi)\to\RR^n:T\mapsto \varphi\circ\xi_T(y)$$ is continuous.  This extends to a topology $\mathcal{T}$ on all of $O^\mathcal{F}_x$, which matches on overlaps (see \cite{sniatycki}).

\mute{

\begin{remark}\labell{r:vectdiffeol}
The \emph{$D$-topology} on a diffeological space $(X,\mathcal{D})$ is the strongest topology on $X$ such that all plots are continuous.  Let $\mathcal{D}_\mathcal{F}$ be the diffeology on $S$ generated by the maps $T\mapsto\xi_T(x)$ for all $X_1,...,X_k\in\mathcal{F}$, $\xi=(X_1,...,X_k)$, and $x\in S$.  Then the $D$-topology generated by $\mathcal{D}_\mathcal{F}$ on $S$ induces the topology described above on each orbit.
\end{remark}

}

\begin{lemma}\labell{l:orbitaltop}
With respect to the topology $\mathcal{T}$, the orbits are connected and pairwise disjoint.
\end{lemma}

\begin{proof}
Fix $x\in S$, and choose $y\in O^\mathcal{F}_x$.  Then there exist $X_1,...,X_k\in\mathcal{F}$ and $t_1,...,t_k\in\RR$ such that $$y=\exp(t_kX_k)\circ...\circ\exp(t_1X_1)(x).$$  Let $T=(t_1,...,t_k)$ and $\xi=(X_1,...,X_k)$.  Then $$y=\xi_T(x).$$  Since $T\mapsto\xi_T(x)$ is continuous with respect to $\mathcal{T}$, and $U_x(\xi)$ is connected, its image is connected.  Hence $x$ and $y$ are in the same connected component of $S$ with respect to $\mathcal{T}$.\\

We now show that each orbit is open and closed with respect to $\mathcal{T}$.  Since the preimage of any orbit via the maps $T\mapsto\xi_T(x)$ is open, each orbit is open in the strongest topology such that each map $\rho_{\xi,x}$ is continuous.  Moreover, since the complement of any orbit is the union of orbits, and hence open, each orbit is closed.
\end{proof}

\begin{example}[Irrational Flow on Torus]
Let $S$ be the torus $\RR^2/\ZZ^2$ and let $\pi:\RR^2\to S$ be the quotient map. Consider the one-element family $\{X\}$ where $X=\pi_*(\partial_1+\sqrt{2}\partial_2)$.  Then for any $x\in S$, $\exp(tX)(x)$ has domain $\RR$, and the orbit is dense in $S$.  $\mathcal{T}$ in this case is such that $O^{\{X\}}_x$ is diffeomorphic to $\RR$.  This is strictly stronger than the subspace topology on the orbit.
\end{example}

\begin{theorem}[Orbit Theorem]\labell{t:singfol}
Let $S$ be a subcartesian space.  Then for any locally complete family of vector fields $\mathcal{F}$, $\mathcal{O}_{\mathcal{F}}$ induces a partition of $S$ into orbits $O^\mathcal{F}_x$, each of which when equipped with the topology $\mathcal{T}$ described above has a smooth manifold structure.  The inclusion $i:O^\mathcal{F}_x\hookrightarrow S$ is smooth, and $i_*:TO^\mathcal{F}_x\to TS$ is a fibrewise linear isomorphism onto $\widehat{T}^\mathcal{F}S|_{O^\mathcal{F}_x}$.
\end{theorem}
\begin{proof}
See \S5 Theorem 3 of \cite{sniatycki}.
\end{proof}

\begin{remark}
This theorem generalises the corresponding ``orbit theorem'' in control theory (see, for example, \cite{jurdjevic}).
\end{remark}

\begin{example}
In \hypref{x:lcexample}{Example} the orbital tangent space has dimension $\dim(\widehat{T}^\mathcal{F}_{(0,y)}\RR^2)=1$ for all $y$, whereas $\widehat{T}^\mathcal{F}_{(x,y)}\RR^2=T_{(x,y)}\RR^2$ for $x\neq0$. But there is only one orbit: all of $\RR^2$.  So the family of vector fields given by the $\RR$-span of $\{\partial_x,x\partial_y\}$ does not satisfy the conclusion of \hypref{t:singfol}{Theorem}.  (Recall that this family is not locally complete.)
\end{example}

\begin{theorem}[Ordering on Orbit Partitions]\labell{t:singfolorder}
Orbits of any family of vector fields $\mathcal{F}$ are contained within orbits of $\Vect(S)$.
\end{theorem}
\begin{proof}
See \S5 Theorem 4 of \cite{sniatycki}.
\end{proof}

\begin{theorem}[Stratification Induced by $\Vect(S)$]\labell{t:stratorb1}
Let $S$ be a smooth stratified space.  Then the orbits on $S$ induced by $\Vect(S)$ form a smooth decomposition of $S$.
\end{theorem}
\begin{proof}
See \S6 Theorem 8 of \cite{sniatycki}.
\end{proof}

\mute{

\begin{remark}
Note that it is not known whether the induced decomposition satisfies the ``local triviality'' condition of a stratified space.  However, it is known that this decomposition satisfies the Whitney A condition (see \cite{lusala-sniatycki11}).  We recall what this conditions is.  Let $N_1$ and $N_2$ be two submanifolds of $\RR^k$ such that $N_1$ is contained in the closure of $N_2$ in $\RR^k$, and let $(x_i)$ be a sequence of points in $N_2$ with limit $x\in N_1$.  Then the sequence of tangent spaces $T_{x_i}N_2$ converges to a linear subspace $L$ of $T_x\RR^k$.  We say that the pair $(N_1,N_2)$ satisfies the Whitney A condition if $L$ contains $T_xN_1$.
\end{remark}

}

\begin{theorem}[Orbits of Stratified Vector Fields]\labell{t:stratorb2}
Let $S$ be a smooth stratified space.  Then the orbits of $\Vect_{strat}(S)$ are exactly the strata of $S$.
\end{theorem}
\begin{proof}
See \S6 Theorem 12 of \cite{sniatycki}.
\end{proof}

\begin{theorem}[$\Vect(M/G)$ and the Orbit-Type Stratification]\labell{t:stratvect}
Given a compact Lie group $G$ acting on a connected manifold $M$, the strata of the orbit-type stratification on $M/G$ are precisely the orbits in $\mathcal{O}$ induced by $\Vect(M/G)$.
\end{theorem}

\begin{proof}
The proof can be found in \cite{sniatycki} and \cite{lusala-sniatycki}.  The idea is the following.  By \hypref{t:bierstone}{Theorem} the orbit-type stratification on $M/G$ is minimal.  The family of stratified vector fields of this stratification is locally complete by \hypref{p:stratloccompl}{Proposition} and its orbits are the strata by \hypref{t:stratorb2}{Theorem}. By \hypref{t:singfolorder}{Theorem}, these strata lie in orbits of $\Vect(M/G)$.  But, the set of orbits $\mathcal{O}$ induced by $\Vect(M/G)$ themselves form a stratification of $M/G$ by \hypref{t:stratorb1}{Theorem}.  So by minimality, we must have that these two stratifications are equal.
\end{proof}

\begin{proposition}\labell{p:jorbital}
Given a Hamiltonian action of a compact Lie group $G$ on a connected symplectic manifold $(M,\omega)$ with momentum map $\mu$, let $Z$ be the zero set of $\mu$.  The orbits of $\ham(Z/G)$ are the orbit-type strata of $Z/G$.
\end{proposition}

\begin{proof}
Sjamaar and Lerman showed in \cite{lerman-sjamaar} that the maximal integral curves of any Hamiltonian vector field on $Z/G$ is confined to a symplectic stratum.  Moreover, we can construct these vector fields so that their orbits are exactly the connected components of the orbit-type strata of $Z/G$.
\end{proof}

\begin{theorem}\labell{t:regvalue}
If $0\in\g^*$ is a regular value of the momentum map $\mu$, then the orbits induced by $\ham(Z/G)$ are exactly the orbits induced by $\Vect(Z/G)$, which gives a minimal stratification.
\end{theorem}

\begin{proof}
Assume that $0\in\g^*$ is a regular value of $\mu$. Then $Z$ is a $G$-manifold, and by \hypref{t:stratvect}{Theorem} the orbit-type stratification is minimal, and the strata are exactly the orbits induced by $\Vect(Z/G)$.  By \hypref{p:jorbital}{Proposition} the orbits of $\Vect(Z/G)$ and $\ham(Z/G)$ coincide.
\end{proof}

\begin{question}\labell{q:critvalue}
Does the above theorem hold in general?  That is, even if $0\in\g^*$ is a critical value of $\mu$?
\end{question} 
\section{Lie Algebras of Vector Fields}\labell{s:liealgebra}

Our goal for this section is to establish that for a locally compact subcartesian space $S$, $\Vect(S)$ is a Lie algebra under the commutator bracket.  For a subset $A\subseteq S$ we shall denote by $\mathfrak{n}(A)$ the set of functions $\{f\in\CIN(S)~|~f|_{A}=0\}.$  Recall that for a family $\mathcal{F}$ of vector fields on $S$ and $x\in S$, $\widehat{T}^\mathcal{F}_xS$ is the linear subspace of $T_xS$ spanned by all vectors $X|_x$ for $X\in\mathcal{F}$.

\begin{proposition}[Characterisation of Orbital Vectors]\labell{p:charTcheck}
Let $S$ be a subcartesian space and $\mathcal{F}$ a locally complete family of vector fields. Let $x\in S$ and $v\in T_xS$.  Then, $v\in\widehat{T}^\mathcal{F}_xS$ if and only if for every open neighbourhood $U\subseteq O^\mathcal{F}_x$ of $i^{-1}(x)$, where $i$ is the inclusion of $O^\mathcal{F}_x$ into $S$, we have $v(\mathfrak{n}(i(U)))=\{0\}$.
\end{proposition}

\begin{proof}
Let $v\in\widehat{T}^\mathcal{F}_xS$.  Then by \hypref{t:singfol}{Theorem} $v=i_*w$ for some $w\in TO^\mathcal{F}_x$.  For any open neighbourhood $U$ of $i^{-1}(x)$ and for any $f\in\mathfrak{n}(i(U))$, $$vf=w(i^*f)=0.$$
Conversely, let $v\in T_xS$ and let $\varphi:V\to\tilde{V}\subseteq\RR^n$ be a chart about $x$. Then, $\varphi(V\cap i(O^\mathcal{F}_x))$ is a differential subspace of $\RR^n$, and in fact since $\varphi\circ i|_{i^{-1}(V)}$ is smooth with $d(\varphi\circ i|_{i^{-1}(V)})$ one-to-one (by \hypref{t:singfol}{Theorem}), we have that $\varphi\circ i|_{i^{-1}(V)}$ is an immersion.  Hence by the rank theorem there exists an open neighbourhood $U\subseteq i^{-1}(V)$ of $i^{-1}(x)$ such that $\tilde{U}:=\varphi\circ i(U)$ is an embedded submanifold of $\RR^n$.\\

Now, $v$ has a unique extension to a vector $\tilde{v}=\varphi_*v\in T_x\RR^n$.  Suppose $vf=0$ for all $f\in \mathfrak{n}(i(U))$.  Then for each such $f$, by \hypref{p:charTS}{Proposition}, $\tilde{v}\tilde{f}=0$ for any local representative $\tilde{f}$ of $f$. But then, also by \hypref{p:charTS}{Proposition}, we have that $\tilde{v}$ is the unique local extension of a vector $\tilde{w}\in T_{\varphi(x)}\tilde{U}$ since $\tilde{f}|_{\tilde{U}}=0$. Since $\tilde{U}$ is an embedded submanifold, there exists a unique $w\in T_{i^{-1}(x)}U$ such that $(\varphi\circ i)_*w=\tilde{w}$.  Identify $\tilde{w}$ with $\tilde{v}$.  By \hypref{t:singfol}{Theorem} and uniqueness, $i_*w=v$.  Thus, $v\in\widehat{T}^\mathcal{F}_xS$.\\

Since any open neighbourhood $W$ of $i^{-1}(x)$ contains a smaller open neighbourhood $U\subseteq i^{-1}(V)\cap W$ in which $\varphi\circ i(U)$ is an embedded submanifold of $\RR^n$, and also $\mathfrak{n}(i(W))\subseteq\mathfrak{n}(i(U))$, we can apply the above argument, obtaining our result.
\end{proof}

\begin{proposition}[Characterisation of Vector Fields]\labell{p:charvect}
Let $S$ be a locally compact subcartesian space.  A derivation $X\in\Der\CIN(S)$ is a vector field if and only if for every $x\in S$ and every open neighbourhood $U$ of $i^{-1}(x)$, $$X(\mathfrak{n}(i(U)))\subseteq\mathfrak{n}(i(U)).$$
\end{proposition}

\begin{proof}
Let $X$ be a vector field.  Then for any $x\in S$ and any open neighbourhood $U$ of $i^{-1}(x)$, $X|_{i(U)}$ is a vector field on $i(U)$.  By \hypref{p:charTcheck}{Proposition} for any $f\in\mathfrak{n}(i(U))$, $$(Xf)|_{i(U)}=0.$$
Conversely, let $X$ be a derivation of $\CIN(S)$ satisfying the property that for any open neighbourhood $U$ of $i^{-1}(x)$, $X(\mathfrak{n}(i(U)))\subseteq\mathfrak{n}(i(U))$ for all orbits $O_x$ with inclusion $i:O_x\to S$.  By \hypref{p:opendomain}{Proposition}, it is enough to show that each maximal integral curve of $X$ has an open domain.\\

Assume otherwise: there exists a maximal integral curve $\exp(tX)(x)$ through a point $x\in S$ with a closed or half-closed domain $I^X_x$.  If $X|_x=0$, then $\exp(tX)(x)$ is a constant map, and its maximal integral curve has $\RR$ as its domain, which is open. So assume $X|_x\neq0$.  Let $a\in I^X_x$ be an endpoint of $I^X_x$ and let $y:=\exp(aX)(x)$.  Then for any open neighbourhood $U\subseteq O_y$ of $i^{-1}(y)$, $$(Xf)|_{i(U)}=0$$ for all $f\in \mathfrak{n}(i(U))$.  In particular, $X|_zf=0$ for all $f\in \mathfrak{n}(i(U))$ and all $z\in i(U)$.  By \hypref{p:charTcheck}{Proposition}, $X|_z\in\widehat{T}_zS$ for all $z\in i(U)$.  Note that since $X|_x\neq0$, we have that $X|_y\neq0$, and so there exists an open neighbourhood $V\subseteq i(U)$  of $y$ such that $X|_z\neq 0$ for all $z\in V$.\\

Since $X|_V$ is a smooth section of $TV\subseteq TS$, by \hypref{t:singfol}{Theorem} we have constructed a vector field $Y\in\Vect(V)$ such that $Y|_z=X|_z$.  But note that by \hypref{p:intcurve}{Proposition} these integral curves locally are restrictions of integral curves in $\RR^n$, and so we can apply the ODE theorem, and obtain that since $X|_V=Y$, we have $\exp(tX)(y)=\exp(tY)(y)$ for $t$ in some domain $I_y$.  But, shrinking $V$ if necessary so that it is an embedded submanifold of $S$ (which exists by the rank theorem), since $Y$ is a vector field on the manifold $V$, $I_y$ is open and contains $0$, whereas since $\exp(tX)(y)=\exp((t+a)(X))(x)$, by assumption $I_y$ has $0\in I_y$ as an endpoint.  This is a contradiction.  Thus, $I^X_x$ does not contain any endpoints, and hence is open.
\end{proof}

\begin{corollary}[$\Vect(S)$ is a Lie Algebra]\labell{c:charvect}
Let $S$ be a locally compact subcartesian space.  Then $\Vect(S)$ is a Lie subalgebra of $\Der\CIN(S)$ and is a $\CIN(S)$-module.
\end{corollary}

\begin{proof}
Let $x\in S$, $X,Y\in\Vect(S)$, $U\subseteq O_x$ any open neighbourhood of $i^{-1}(x)$ and $f\in\mathfrak{n}(i(U))$ and $g\in\CIN(S)$.  Applying \hypref{p:charvect}{Proposition}, we have $(X+Y)f|_{i(U)}=Xf|_{i(U)}+Yf|_{i(U)}=0$, $(gX)f|_{i(U)}=0$ and $[X,Y](f)|_{i(U)}=X(Yf)|_{i(U)}-Y(Xf)|_{i(U)}=0$.  Thus, $X+Y$, $gX$ and $[X,Y]$ are vector fields.
\end{proof}

\begin{remark}
By the above corollary, for any $x\in S$ and any $v\in\widehat{T}_xS$, there is a vector field $X$ such that $X|_x=v$.  In other words, we did not need to take the linear span in the definition of $\widehat{T}_xS$.
\end{remark}

We again return to the situation of a Hamiltonian $G$-action on $(M,\omega)$.  We have shown that $\Vect(Z)$ is a Lie algebra.  Denote by $\Vect(Z)^G$ the Lie subalgebra of $G$-invariant vector fields on $Z$.

\begin{proposition}[Invariant Local Extensions for $\Vect(Z)^G$]\labell{p:vectZext}
Let $X\in\Vect(Z)^G$ and let $x\in Z\subseteq M$.  Then there exist a $G$-invariant open neighbourhood $U\subseteq M$ of $x$ and $\tilde{X}\in\Vect(M)^G$ such that $$X|_{U\cap Z}=\tilde{X}|_{U\cap Z}.$$
\end{proposition}

\begin{proof}
There exist an open neighbourhood $V\subseteq M$ of $x$ and $\tilde{X}_0\in\Vect(M)$ such that $\tilde{X}_0|_{V\cap Z}=X|_{V\cap Z}$.  Let $g_0=e\in G$ and let $g_i$ be elements of $G$ for $i=1,...,k$ such that $G\cdot x\subseteq M$ is covered by open sets $g_i\cdot V$.  Let $\{\zeta_i\}$ be a partition of unity subordinate to this cover, and define $$\tilde{X}:=\sum_{i=0}^k\zeta_i g_{i*}\tilde{X}_0.$$  Then, letting $W:=\bigcup_{i=0}^kg_i\cdot V$, we have that for any $y\in W\cap Z$
\begin{align*}
\tilde{X}|_y=&~\sum_{i=0}^k\zeta_i(y)g_{i*}(\tilde{X}_0|_{g_i^{-1}\cdot y})\\
=&~\sum_{i=0}^k\zeta_i(y)g_{i*}(X|_{g_i^{-1}\cdot y})\\
=&~\sum_{i=0}^k\zeta_i(y)X|_y\\
=&~X|_y.
\end{align*}
Thus, $\tilde{X}\in\Vect(M)$ is a local extension of $X$ on $W\cap Z$.  Averaging $\tilde{X}$ and letting $U$ be a $G$-invariant open neighbourhood of $G\cdot x$ contained in $W$, we are done.
\end{proof}

\begin{proposition}[$\Vect(Z)^G$ is a Lie Algebra]\labell{p:vectZloccompl}
$\Vect(Z)^G$ is a locally complete Lie subalgebra of $\Vect(Z)$.
\end{proposition}

\begin{proof}
Since diffeomorphisms commute with the commutator bracket, we have that $\Vect(Z)^G$ is a Lie subalgebra of $\Vect(Z)$.  For any two invariant vector fields $X$ and $Y$, we have for all $g\in G$ and $x\in Z$ $$g\cdot\exp(tX)(\exp(sY)(x))=\exp(tX)(\exp(sY)(g\cdot x))$$ for $s,t$ such that the composition of the curves is defined. Thus $\exp(tX)_*Y$ is locally defined about $G$-orbits.  Since $\Vect(Z)$ is locally complete, for any $x\in Z$ there exist a vector field $\Xi$ on $Z$ and an open neighbourhood $U$ of $x$ such that $\exp(tX)_*Y$ is defined on $U$ and $(\exp(tX)_*Y)|_U=\Xi|_U$.  Since $\exp(tX)_*Y$ is invariant about $x$, we can choose $U$ to be a $G$-invariant open neighbourhood.  Let $V\subset U$ be a $G$-invariant open neighbourhood of $x$ such that $\overline{V}\subset U$.  Let $b:M\to\RR$ be a $G$-invariant smooth bump function with support in $U$ and $b|_V=1$.  Then, $b\Xi\in\Vect(Z)^G$ extends $(\exp(tX)_*Y)|_V$ to a invariant vector field on $Z$.
\end{proof}

\begin{definition}
Let $\rho_Z:\g\to\Der\CIN(Z)$ be the $\g$-action induced by the action of $G$ on $Z$.  Note that by \hypref{p:opendomain}{Proposition}, $\rho_Z(\g)\subseteq\Vect(Z)$.  In fact, for any $\xi\in\g$, $\xi_Z:=\rho_Z(\xi)$ is just the restriction of $\xi_M$ to $Z$.
\end{definition}

\begin{proposition}[$\rho_Z(\g)$ is a Lie Algebra]\labell{p:rhoZloccompl}
$\rho_Z(\g)$ is a locally complete Lie subalgebra of $\Vect(Z)$.
\end{proposition}

\begin{proof}
Let $\xi,\zeta\in\g$, and let $\xi_Z=\rho_Z(\xi)$ and $\zeta_Z=\rho_Z(\zeta)$.  Then, $\exp(t\xi_Z)_*\zeta_Z=(\Ad_{\exp(t\xi)}\zeta)_Z$.  Thus $\rho_Z(\g)$ is locally complete, and since $\rho_Z$ is a Lie algebra homomorphism, its image is a Lie algebra.
\end{proof}

\begin{corollary}
$\rho_Z([\g,\g])$ and $\rho_Z(\mathfrak{z}(\g))$ are both locally complete Lie subalgebras of $\Vect(Z)$.
\end{corollary}

\begin{proof}
This is immediate from the above lemma.
\end{proof}

\begin{definition}\labell{d:AZ}
Define $\mathcal{A}_Z$ to be the smallest Lie subalgebra of $\Vect(Z)$ that contains both $\rho_Z(\g)$ and $\Vect(Z)^G$.
\end{definition}

\begin{proposition}\labell{p:AZ}
$\mathcal{A}_Z$ is locally complete and is equal to the direct sum of Lie subalgebras $$\mathcal{A}_Z=\rho_Z([\g,\g])\oplus\Vect(Z)^G.$$
\end{proposition}

\begin{proof}
By \hypref{p:vectZext}{Proposition}, for any $X\in\Vect(Z)^G$ and for any $x\in Z$, there exist a $G$-invariant open neighbourhood $U\subseteq M$ of $x$ and $\tilde{X}\in\Vect(M)^G$ such that $$X|_{U\cap Z}=\tilde{X}|_{U\cap Z}.$$  Hence, $$[\xi_Z,X]|_{U\cap Z}=[\xi_M,\tilde{X}]|_{U\cap Z}=0$$ by \hypref{p:A}{Proposition}.  Thus, applying \hypref{p:intcurve}{Proposition} and \hypref{e:flowscommute}{Equation} in the proof of \hypref{p:A}{Proposition}, we have that
\begin{equation}\labell{e:flowscommute2}
\exp(t\xi_Z)\circ\exp(sX)=\exp(sX)\circ\exp(t\xi_Z).
\end{equation}

Now, let $\xi\in\g$ and assume for all $g\in G$ and $x\in Z$, we have $$g_*(\xi_Z|_x)=\xi_Z|_{g\cdot x}.$$ Then, $$\frac{d}{dt}\Big|_{t=0}(g\cdot\exp(t\xi_Z)(x))=\frac{d}{dt}\Big|_{t=0}\exp(t\xi_Z)(g\cdot x).$$  The uniqueness property of $\exp$ implies that $$g\cdot\exp(t\xi_Z)(x)=\exp(t\xi_Z)(g\cdot x).$$  Hence  $(g\exp(t\xi))\cdot x=(\exp(t\xi)g)\cdot x$.  Since this is true for all $g\in G$, $\exp(t\xi)$ must be in the centre of $G$, and hence $\xi\in\mathfrak{z}(\g)$.  Thus, $$\rho_Z(\g)\cap\Vect(Z)^G=\rho_Z(\mathfrak{z}(\g)).$$  Since $\rho_Z$ is a Lie algebra homomorphism, from \hypref{e:flowscommute2}{Equation}: $\rho_Z(\g)=\rho_Z([\g,\g])\oplus\rho_Z(\mathfrak{z}(\g))$, and we obtain the direct sum structure of $\mathcal{A}_Z$.\\

To show local completeness, by \hypref{p:vectZloccompl}{Proposition} and \hypref{p:rhoZloccompl}{Proposition} it suffices to show that for any $\xi\in\g$ and $X\in\Vect(Z)^G$, $\exp(t\xi_Z)_*X\in\mathcal{A}_Z$ and $\exp(tX)_*\xi_Z\in\mathcal{A}_Z$.  The former is immediate since $X$ is invariant.  The latter follows from \hypref{e:flowscommute2}{Equation}:
\begin{align*}
\exp(tX)_*(\xi_Z|_x)=&~\frac{d}{ds}\Big|_{s=0}\exp(tX)(\exp(s\xi_Z)(x))\\
=&~\frac{d}{ds}\Big|_{s=0}\exp(s\xi_Z)(\exp(tX)(x))\\
=&~\xi_Z|_{\exp(tX)(x)}.
\end{align*}
\end{proof}

\begin{proposition}[$\ham(Z/G)$ is a Lie Algebra]
$\ham(Z/G)$ is a locally complete Lie subalgebra of $\Vect(Z/G)$.
\end{proposition}

\begin{proof}
For any $f,g,h\in\CIN(Z/G)$ and $a,b\in\RR$, $\{af+bg,h\}_{Z/G}=a\{f,h\}_{Z/G}+b\{g,h\}_{Z/G}$, and so $aX_f+bX_g=X_{af+bg}$.  Thus $\ham(Z/G)$ is a real vector space.  Next, the Jacobi identity for the Poisson bracket gives $$\{\{f,g\}_{Z/G},h\}_{Z/G}=-\{g,\{f,h\}_{Z/G}\}_{Z/G}+\{f,\{g,h\}_{Z/G}\}_{Z/G}.$$ This translates to $$X_{\{f,g\}_{Z/G}}h=X_fX_gh-X_gX_fh=[X_f,X_g]h.$$

To show local completeness, fix $f,g\in\CIN(Z/G)$ and let $X_f$ and $X_g$ be their corresponding Hamiltonian vector fields.  For sufficiently small $t$, we want to show that $\exp(tX_f)_*X_g$ is a Hamiltonian vector field.  Consider $X_{\exp(-tX_f)^*g}$.  For any $h\in\CIN(Z/G)$, we have
\begin{align*}
X_{\exp(-tX_f)^*g}h=&~\pois{\exp(-tX_f)^*g}{h}_{Z/G}\\
=&~\exp(-tX_f)^*\pois{g}{\exp(tX_f)^*h}\\
=&~\exp(-tX_f)^*(X_g(\exp(tX_f)^*h))\\
=&~(\exp(tX_f)_*X_g)(h).
\end{align*}
This completes the proof.
\end{proof} 
\section{Orbital Maps}\labell{s:orbitalmaps}

In general, a smooth map between subcartesian spaces does not lift to a map between the corresponding orbital tangent bundles.  This is illustrated in the following example.

\begin{example}
Let $S=\{(x,y)\in\RR^2~|~xy=0\}$, and let $\gamma:\RR\to S$ be a curve passing through $(0,0)\in S$ at time $t=0$ such that $$u:=\frac{d}{dt}\Big|_{t=0}\gamma(t)\neq0.$$  Then $u\notin\widehat{T}_{(0,0)}S$ since $\widehat{T}_{(0,0)}S=\{0\}$, but $\frac{d}{dt}\Big|_{t=0}\in\widehat{T}_0\RR=T_0\RR$.
\end{example}

To remedy this lack of the functoriality of $\widehat{T}$, we introduce a special kind of smooth map.

\begin{definition}[Orbital Maps]
Let $R$ and $S$ be subcartesian spaces and let $F:R\to S$ be a smooth map between them.  Let $\mathcal{F}$ and $\mathcal{G}$ be families of vector fields on $R$ and $S$, respectively.  $F$ is \emph{orbital} with respect to $\mathcal{F}$ and $\mathcal{G}$ if for any $x\in R$, $F(O^\mathcal{F}_x)\subseteq O^\mathcal{G}_{F(x)}$.  That is, for any $X\in\mathcal{F}$, $x\in R$, and $t\in I_x^X$, there exist $Y_1,...,Y_k\in\mathcal{G}$ and $t_1,...,t_k\in\RR$ such that
\begin{equation*}
F(\exp(tX)(x))=\exp(t_kY_k)\circ...\circ\exp(t_1Y_1)(F(x)).
\end{equation*}
If $\mathcal{F}=\Vect(R)$ and $\mathcal{G}=\Vect(S)$, then we simply call $F$ \emph{orbital}.
\end{definition}

\begin{proposition}[Charts, Smooth Functions, Diffeomorphisms]\labell{p:orb}
Charts, real-valued smooth functions, and diffeomorphisms between subcartesian spaces are orbital.
\end{proposition}

\begin{proof}
Since $\RR^k$ only has one orbit for each $k\geq0$, charts and smooth functions are trivially orbital.  Since a diffeomorphism $F:R\to S$ induces an isomorphism of Lie algebras $F_*:\Der\CIN(R)\to\Der\CIN(S)$, and hence $F(\exp(tX)(x))=\exp(tF_*X)(F(x))$ for all $X\in\Vect(R)$ and $x\in R$, we are done.
\end{proof}

\begin{proposition}[Orbital Pushforwards]\labell{p:Torbital}
Let $R$ and $S$ be subcartesian spaces, and let $F$ be an orbital map between them with respect to locally complete families of vector fields $\mathcal{F}$ on $R$ and $\mathcal{G}$ on $S$.  Then the restriction of the pushforward $F_*$ to $\widehat{T}^\mathcal{F}R$ has image in $\widehat{T}^\mathcal{G}S$.
\end{proposition}

\begin{proof}
This is immediate from \hypref{t:singfol}{Theorem} and the definition of an orbital map.
\end{proof}

\begin{remark}
Subcartesian spaces equipped with locally complete families of vector fields, along with orbital maps with respect to these families, form a category.  We will call the objects of this category \emph{orbital subcartesian spaces}.
\end{remark}

\mute{
Recall the $D$-topology on a subcartesian space induced by a family of vector fields, as described in \hypref{r:vectdiffeol}{Remark}.

\begin{proposition}
Let $R$ and $S$ be subcartesian spaces, and let $\mathcal{F}$ and $\mathcal{G}$ be locally complete families of vector fields on each, respectively.  Then a smooth map $F:R\to S$ is orbital if and only if it is diffeologically smooth with respect to the diffeologies $\mathcal{D}_\mathcal{F}$ and $\mathcal{D}_\mathcal{G}$.
\end{proposition}

\begin{proof}
Assume that $F$ is orbital with respect to $\mathcal{F}$ and $\mathcal{G}$ and fix a plot $p\in\mathcal{D}_\mathcal{F}$.  We wish to show that $F\circ p\in\mathcal{D}_\mathcal{G}$.  Let $U$ be the domain of $p$, and fix $u\in U$.  By definition of $\mathcal{D}_\mathcal{F}$ there exist an open neighbourhood $V\subseteq U$ of $u$, vector fields $X_1,...,X_k\in\mathcal{F}$, $x\in R$, and a smooth map $f:V\to U_x(\xi)$ where $\xi=(X_1,...,X_k)$ such that $$p|_V(v)=\xi_{f(v)}(x).$$ It is enough to show that $v\mapsto F(\xi_{f(v)}(x))$ is smooth.  Let $T=(t_1,...,t_k)=f(v)$.  Since $F$ is orbital, there exist $Y_1,...,Y_l\in\mathcal{G}$ and $s_1,...,s_l\in\RR$ such that
\begin{align*}
F(\xi_T(x))=&~F(\exp(t_kX_k)\circ...\circ\exp(t_1X_1)(x))\\
=&~\exp(s_lY_l)\circ...\circ\exp(s_1Y_1)\circ{F(x)}\\
=&~\zeta_{T'}(F(x))
\end{align*}
where $\zeta=(Y_1,...,Y_l)$ and $T'=(s_1,...,s_l)$. Thus, locally about $u$, $F\circ p$ is equal to $\zeta_{f(v)}(F(x))$, and so is a plot in $\mathcal{D}_\mathcal{G}$.\\

Now assume that $F$ is diffeologically smooth with respect to the diffeologies $\mathcal{D}_\mathcal{F}$ and $\mathcal{D}_\mathcal{G}$.  Fix $x\in R$.  We want to show that $F(O_x^{\mathcal{F}})\subseteq O^{\mathcal{G}}_{F(x)}$.  Fixing $X\in\mathcal{F}$, if $I^X_x$ is the domain of $\exp(\cdot X)(x)$, then it is enough to show that for any $t\in I^X_x$, $F(\exp(tX)(x))\in O^\mathcal{G}_{F(x)}$.  Fix $t$.  If we let $\xi=X$ and $T=t\in I^X_x$, then $\exp(tX)(x)=\xi_T(x)$, and $T\mapsto\xi_T(x)$ is a plot in $\mathcal{D}_\mathcal{F}$.  Thus, $t\mapsto F(\exp(tX)(x))$ is a plot in $\mathcal{D}_\mathcal{G}$.  Thus, for $\tau$ near $t$, we have that $\tau\mapsto F(\exp(\tau X)(x))$ is equal to $\tau\mapsto\zeta_\tau(F(x))$ for some $\zeta=(Y_1,...,Y_k)$ where $Y_1,...,Y_k\in\mathcal{G}$.  But the image of this plot is thus contained in the same orbit as $F(x)$, and we are done.
\end{proof}

}

\begin{proposition}
Let $R$ and $S$ be smooth stratified spaces.  Then a smooth map $F:R\to S$ is stratified if and only if it is orbital with respect to $\Vect_{strat}(R)$ and $\Vect_{strat}(S)$.
\end{proposition}

\begin{proof}
This is immediate from \hypref{t:stratorb2}{Theorem}.
\end{proof}

\begin{corollary}\labell{c:orbitalcat}
The category of smooth stratified spaces, along with smooth stratified maps, forms a full subcategory of orbital subcartesian spaces.
\end{corollary}

The following theorem is a result of Schwarz; see \cite{schwarz2} and \cite{schwarz3} (\cite{schwarz3} Chapter 1 Theorem 4.3 for full details).  Let $D$ be the Lie subgroup of $\Diff(M)^G$ consisting of $G$-equivariant diffeomorphisms of $M$ that act trivially on $\CIN(M)^G$ (that is, they send each $G$-orbit to itself), and let $\mathfrak{d}$ denote the Lie algebra of $D$.

\begin{notation}
For brevity, we will often use the notation $\mathcal{V}:=\Vect(M)^G$ in the future.
\end{notation}

\begin{theorem}[Schwarz]\labell{t:vectquot}
The following is a split short exact sequence.
\begin{equation}\labell{e:ses}
\xymatrix{
0 \ar[r] & \mathfrak{d} \ar[r] & \Vect(M)^G \ar[r]^{\pi_*} & \Vect(M/G) \ar[r] & 0
}
\end{equation}
\end{theorem}

\begin{remark}
Actually, Schwarz showed that $\pi_*$ mapped $\Vect(M)^G$ onto stratified vector fields of $M/G$ with its orbit-type stratification. But by \hypref{t:stratvect}{Theorem}, this family of vector fields is exactly $\Vect(M/G)$.
\end{remark}

\begin{remark}
Since diffeomorphisms in $D$ keep $G$-orbits invariant, we have $\widehat{T}^{\mathfrak{d}}M\subseteq\widehat{T}^{\rho(\g)}M$. In fact, if $G$ is abelian then we have $\rho(\g)\subset\mathfrak{d}$, and so we obtain $$\widehat{T}^{\mathfrak{d}}M=\widehat{T}^{\rho(\g)}M.$$  However, in the non-abelian case, $\widehat{T}^{\mathfrak{d}}M$ may be a strict subset of $\widehat{T}^{\rho(\g)}M$.  For example, consider $\SO(3)$ acting by rotations on $\RR^3$.  For any nonzero $x\in\RR^3$ and any nonzero $\xi\in\mathfrak{so}(3)$, $\xi_{\RR^3}|_x$ is tangent to the $\SO(3)$-orbit through $x$, but this vector is not in the image of any invariant vector field.  For if it was, then the stabiliser at $x$ would fix the vector, and this is not the case.
\end{remark}

\begin{corollary}\labell{c:vectquot}
The image of $\pi_*$ restricted to $\widehat{T}^\mathcal{A}M$ is $\widehat{T}(M/G)$.
\end{corollary}

\begin{proof}
$\pi_*$ will map any vector in $\widehat{T}^{\rho(\g)}M$ to 0, and so it is enough to consider vectors in $\widehat{T}^\mathcal{V}M$ (where we set $\mathcal{V}:=\Vect(M)^G$ for brevity).  Let $x\in M$ and $v\in\widehat{T}_x^\mathcal{V}M$.  Then, there exists a invariant vector field $X\in\mathcal{V}$ such that $X|_x=v$.  By \hypref{t:vectquot}{Theorem} there exists $Y\in\Vect(M/G)$ such that $Y|_{\pi(x)}=\pi_*(X|_x)$.\\

Now, let $w\in\widehat{T}_{\pi(x)}(M/G)$.  There exists a vector field $Y\in\Vect(M/G)$ such that $Y|_{\pi(x)}=w$.   Again by \hypref{t:vectquot}{Theorem} there is a vector field $X\in\mathcal{V}$ such that $\pi_*X=Y$, and so $\pi_*(X|_x)=w$.
\end{proof}

\begin{corollary}\labell{c:piorbital}
$\pi$ is orbital with respect to $\mathcal{A}$ and $\Vect(M/G)$.
\end{corollary}

\begin{proof}
Since $\pi_*$ will map any vector field in $\rho(\g)$ to the zero vector field on $M/G$, and local flows of $\Vect(M)^G$ and $\rho(\g)$ commute, it is enough to check that $\pi$ is orbital with respect to $\Vect(M)^G$ and $\Vect(M/G)$.  Let $X\in\Vect(M)^G$.  Then by \hypref{t:vectquot}{Theorem}, there is a vector field $Y\in\Vect(M/G)$ such that $\pi_*X=Y$.  Fix $x\in M$.  Then $$\frac{d}{dt}\pi(\exp(tX)(x))=\pi_*(X|_x)=Y|_{\pi(x)}=\frac{d}{dt}\exp(tY)(\pi(x)).$$
By the ODE theorem, we have that $$\pi(\exp(tX)(x))=\exp(tY)(\pi(x))$$
for all $t$ where it is defined.  Hence orbits in $\mathcal{O}_\mathcal{A}$ are mapped via $\pi$ to orbits of $M/G$ induced by $\Vect(M/G)$.
\end{proof}

\begin{corollary}\labell{c:flowlift}
A local flow of $M/G$ lifts to a $G$-equivariant local flow of $M$.
\end{corollary}

\begin{proof}
Fix a vector field $Y\in\Vect(M/G)$.  By \hypref{t:vectquot}{Theorem} there is a vector field $X\in\Vect(M)^G$ such that $\pi_*X=Y$. From the ODE theorem we have that $$\pi(\exp(tX)(x))=\exp(tY)(\pi(x))$$ for all $x\in M$ and $t\in I^X_x$.
\end{proof}

\begin{theorem}\labell{t:singfolA}
The orbits in $\mathcal{O}_\mathcal{A}$ are exactly the orbit-type strata on $M$.
\end{theorem}

\begin{proof}
Fix $x\in M$, and let $H\leq G$ be a closed subgroup of $G$ such that $x\in M_{(H)}$.  Choose $y\in O^\mathcal{A}_x$.  Then, there exist vector fields $X_1,...,X_k\in\mathcal{A}$ and $t_1,...,t_k\in\RR$ such that $$y=\exp(t_1X_1)\circ...\circ\exp(t_kX_k)(x).$$  But then, by \hypref{c:piorbital}{Corollary} and \hypref{t:vectquot}{Theorem}, there exist $Y_1,...,Y_k\in\Vect(M/G)$ such that $$\pi(y)=\exp(t_1Y_1)\circ...\circ\exp(t_kY_k)(\pi(x)).$$  Hence, $\pi(x)$ and $\pi(y)$ are in the same orbit $O_{\pi(x)}$.  But this is a stratum of the orbit-type stratification of $M/G$ by \hypref{t:stratvect}{Theorem}, and so $y\in M_{(H)}$.  Thus $O^\mathcal{A}_x\subseteq M_{(H)}$.\\

Now, let $z$ be a point in the same connected component of $M_{(H)}$ as $x$.  Then again by \hypref{t:stratvect}{Theorem}, $\pi(y)$ and $\pi(x)$ are in the same orbit $O_{\pi(x)}$, and hence there exist vector fields $Y_1,...,Y_k$ and $t_1,...,t_k\in\RR$ such that $\pi(y)=\exp(t_1X_1)\circ...\circ\exp(t_kX_k)(\pi(x))$.  By \hypref{c:flowlift}{Corollary}, there are vector fields $X_1,...,X_k\in\mathcal{A}$ such that $y=\exp(t_1X_1)\circ...\circ\exp(t_kX_k)(x)$.
\end{proof}

We again return to the case where $G$ is a compact Lie group now acting on a connected symplectic manifold $(M,\omega)$ in a Hamiltonian fashion, with $Z$ the zero set of the momentum map $\mu$.

$$\xymatrix{
Z \ar[r]^i \ar[d]_{\pi_Z} & M \ar[d]^{\pi} \\
Z/G \ar[r]_j & M/G \\
}$$

Recall that $\mathcal{A}=\rho(\g)+\Vect(M)^G$ and $\mathcal{A}_Z=\rho_Z(\g)+\Vect(Z)^G$ (see \hypref{d:A}{Definition} and \hypref{d:AZ}{Definition}).

\begin{proposition}\labell{p:iorbit}
$i$ is orbital with respect to $\mathcal{A}_Z$ and $\mathcal{A}$.
\end{proposition}

\begin{proof}
Let $X\in\mathcal{A}_Z$ and fix $z\in Z\subseteq M$.  Then by \hypref{p:vectZext}{Proposition} there exist a $G$-invariant open neighbourhood $U\subseteq M$ of $z$ and $\tilde{X}\in\mathcal{A}$ such that $X|_{U\cap Z}=\tilde{U\cap Z}$.  Applying the ODE theorem, we are done.
\end{proof}

\begin{proposition}\labell{p:piZorbit}
$\pi_Z$ is orbital with respect to $\mathcal{A}_Z$ and $\Vect(Z/G)$.
\end{proposition}

\begin{proof}
By \hypref{p:AZ}{Proposition}, it is enough to show this separately for $\rho_Z(\g)$ and $\Vect(Z)^G$.  For the first subalgebra, $$\pi(\exp(t\xi_Z)(z))=\pi(z)=\exp(0)(\pi(z))$$ for all $z\in Z$ and $t$ for which the integral curve is defined.\\

Now fix $X\in\Vect(Z)^G$.  Using \hypref{p:vectZext}{Proposition} cover $Z/G$ with a locally finite open cover $\{V_\alpha\}_{\alpha\in A}$ such that for every $\alpha\in A$, there exist $\tilde{X}^\alpha\in\Vect(M)^G$ satisfying $i_*(X|_{\pi_Z^{-1}(V_\alpha)})=\tilde{X}^\alpha|_{\pi^{-1}(j(V_\alpha))}.$ Note that for any $\alpha\in A$, $x\in V_\alpha$, $z\in\pi_Z^{-1}(x)$ and $f\in\mathfrak{n}(j(Z/G))$,
\begin{align*}
(\pi_*\tilde{X}^\alpha)|_{j(x)}f=&~\tilde{X}^\alpha|_{i(z)}\pi^*f\\
=&~X|_zi^*\pi^*f\\
=&~X|_z\pi_Z^*j^*f\\
=&~0.
\end{align*}
Let $\{\zeta_\alpha\}_{\alpha\in A}$ be a partition of unity subordinate to $\{V_\alpha\}$, and for each $\alpha\in A$, let $\tilde{\zeta}_\alpha$ be an extension of $\zeta_{\alpha}$ to $M/G$.  Define $$\tilde{Y}:=\sum_{\alpha}(\tilde{\zeta}_{\alpha}(\pi_*\tilde{X}^\alpha))|_{j(Z/G)}.$$ From the above, we have that $\tilde{Y}(f)=0$ for all $f\in\mathfrak{n}(j(Z/G))$, and so in particular, $\tilde{Y}$ restricts to a global derivation $Y\in\Der\CIN(S)$.  Also, for any $z\in Z$,
\begin{align*}
j_*\pi_{Z*}(X|_z)=&~\sum_{\alpha}\tilde{\zeta}_\alpha j_*\pi_{Z*}(X|_z)\\
=&~\sum_{\alpha}\tilde{\zeta}_\alpha\pi_*(\tilde{X}^\alpha|_{i(z)})\\
=&~\tilde{Y}|_{\pi(i(z))}\\
=&~j_*Y|_{\pi_Z(z)}.
\end{align*}
Thus, $\pi_{Z*}(X|_z)=Y|_{\pi_Z(z)}.$ Finally, we need to show that $Y$ is a vector field, and we shall do so by appealing to \hypref{p:opendomain}{Proposition}.  Fix $z\in Z$, and define $\gamma(t):=\pi_Z(\exp(tX)(z))$.  Differentiating, we see that $\gamma$ is an integral curve of $Y$ through $\pi_Z(z)$.  But $\gamma$ has an open domain and $\pi_Z$ is surjective, and so $\gamma$ is maximal.  Thus $Y$ is a vector field.
\end{proof}

\begin{proposition}\labell{p:jorbital2}
$j$ is orbital with respect to $\ham(Z/G)$ and $\Vect(M/G)$.
\end{proposition}

\begin{proof}
By \hypref{p:jorbital}{Proposition} orbits of $\ham(Z/G)$ are exactly the orbit-type strata of $Z/G$, which in turn are contained in the orbit-type strata of $M/G$. By \hypref{t:stratvect}{Theorem}, connected components of the orbit-type strata of $M/G$ are the orbits induced by $\Vect(M/G)$.
\end{proof}

\begin{lemma}\labell{l:GhamtangZ}
Vector fields in $\ham(M)^G$ are tangent to level sets of $\mu$.
\end{lemma}

\begin{proof}
Fix $X\in\ham(M)^G$.  There exists $f\in\CIN(M)^G$ such that $X=X_f$.  It is enough to show that for any $\xi\in\g$, we have $X(\mu^\xi)=0$.  Fix $\xi\in\g$.  Then
\begin{align*}
X(\mu^\xi)=&~d\mu^\xi(X)\\
=&~\omega(X,\xi_M)\\
=&~-df(\xi_M)=0.
\end{align*}
\end{proof}

\begin{lemma}\labell{l:hamsurj}
There is a surjective Lie algebra homomorphism $H:\ham(M)^G\to\ham(Z/G)$ sending $X\in\ham(M)^G$ to $(\pi_Z)_*(X|_Z)$.
\end{lemma}

\begin{proof}
Fix $X\in\ham(M)^G$, and let $f\in\CIN(M)^G$ such that $X=X_f$.  By \hypref{l:GhamtangZ}{Lemma}, we have that $X|_Z$ is tangent to $Z$.  By \hypref{p:opendomain}{Proposition}, since the integral curves of $X$ through points of $Z$ are contained in $Z$, these integral curves when restricted to $Z$ have open domains, and hence $X|_Z\in\Vect(Z)^G$.\\

We now need to show that $(\pi_Z)_*(X|_Z)$ is a smooth vector field on $Z/G$.  Define $$h:=j^*((\pi^*)^{-1}(f)).$$  We claim that $X_h\in\ham(Z/G)$ is exactly $(\pi_Z)_*(X|_Z)$.  By \hypref{p:jorbital}{Proposition} it is enough to show this on each stratum $(Z/G)_{(H)}$ of $Z/G$.  Since $X$ is $G$-invariant, it is in fact tangent to $Z_{(H)}$ by \hypref{t:singfolA}{Theorem} and \hypref{l:GhamtangZ}{Lemma} for each $H\leq G$.  Fix a nonempty $Z_{(H)}$.  Then  $Y:=(\pi_{(H)})_*(X|_{Z_{(H)}})$ is a smooth vector field on $(Z/G)_{(H)}$.  Let $g\in\CIN(M/G)$ such that $\pi^*g=f$.  We have
\begin{align*}
\pi_{(H)}^*(Y\hook\omega_{(H)})=&~X_f|_{Z_{(H)}}\hook i_{(H)}^*\omega\\
=&~i_{(H)}^*(X_f\hook\omega)\\
=&~i_{(H)}^*(-df)\\
=&~(\pi\circ i_{(H)})^*(-dg)\\
=&~(j\circ\pi_{(H)})^*(-dg)\\
=&~\pi_{(H)}^*(-dj^*g|_{(Z/G)_{(H)}})\\
=&~\pi_{(H)}^*(-d(h|_{(Z/G)_{(H)}})).
\end{align*}

Now, since $Z_{(H)}$ is a $G$-manifold with quotient manifold $(Z/G)_{(H)}$, it is known that $\pi_{(H)}^*$ is an isomorphism of complexes between differential forms on $(Z/G)_{(H)}$ and basic differential forms on $Z_{(H)}$.  Hence, $$Y\hook\omega_{(H)}=-d(h|_{(Z/G)_{(H)}}).$$  Thus, $Y=X_h|_{(Z/G)_{(H)}}$.  Thus the map $H$ is well-defined.\\

To show that this map is surjective, it is enough to show that there is a surjective map sending $f\in\CIN(M)^G$ to $j^*((\pi^*)^{-1}(f))$.  But $\pi^*$ is an isomorphism between $\CIN(M/G)$ and $\CIN(M)^G$, and since $Z/G$ is closed in $M/G$, we have that $j^*$ is a surjection from $\CIN(M/G)$ onto $\CIN(Z/G)$ by \hypref{p:closedsubset}{Proposition}.\\

We now check that this is a Lie algebra homomorphism.  It is clearly $\RR$-linear.  Let $f,g\in\CIN(M)^G$.  Then $$(\pi^*)^{-1}(\pois{f}{g})=\pois{(\pi^*)^{-1}f}{(\pi^*)^{-1}g}_{M/G},$$ and $j^*$ is a Poisson morphism.  Thus, $$H(X_{\pois{f}{g}})=\pois{H(X_f)}{H(X_g)}_{Z/G}.$$
\end{proof}

\begin{proposition}\labell{p:vectZorbits}
The orbits of $\mathcal{A}_Z$ are contained in the orbit-type strata of $Z$.  Moreover, if $G$ is connected, the orbits are exactly the orbit-type strata.
\end{proposition}

\begin{proof}
By \hypref{p:iorbit}{Proposition}, $i$ is orbital with respect to $\mathcal{A}_Z$ and $\mathcal{A}$.  Thus, orbits of $\mathcal{A}_Z$ are mapped into orbits of $\mathcal{A}$, which by \hypref{t:singfolA}{Theorem} are exactly the orbit-type strata on $M$.  Thus, the orbits of $\mathcal{A}_Z$ are contained in the orbit-type strata on $M$ intersected with $Z$.  But these are precisely the orbit-type strata of $Z$.\\

For the opposite inclusion, assume that $G$ is connected.  Let $x,y$ be in the same orbit-type stratum in $Z$.  Then $\pi_Z(x)$ and $\pi_Z(y)$ are in the same orbit-type stratum in $Z/G$.  Thus by \hypref{p:jorbital}{Proposition}, there exist $f_1,...,f_k\in\CIN(Z/G)$ and $t_1,...,t_k\in\RR$ such that the Hamiltonian vector fields $X_{f_1},...,X_{f_k}$ satisfy $$\pi_Z(y)=\exp(t_1X_{f_1})\circ...\circ\exp(t_kX_{f_k})(\pi_Z(x)).$$
By \hypref{l:hamsurj}{Lemma}, there exist $Y_1,...,Y_k\in\ham(M)^G$ such that $(\pi_Z)_*(Y_i|_Z)=X_{f_i}$ for each $i=1,...,k$.  So, we have $$\pi_Z(y)=\pi_Z(\exp(t_1Y_1|_Z)\circ...\circ\exp(t_kY_k|_Z)(\pi_Z(x)).$$ In particular, $$z:=\exp(t_1Y_1|_Z)\circ...\circ\exp(t_kY_k)(x)$$ is contained in the same $G$-orbit as $y$.  Thus there is some $g\in G$ such that $g\cdot z=y$.  Since $G$ is compact and connected, there is some $\tau\in\RR$ and $\xi\in\g$ such that $y=g\cdot z=\exp(\tau\xi_Z)(z)$.  Thus, $x$ and $y$ are in the same orbit of $\mathcal{A}_Z$.
\end{proof}

\begin{proposition}
If $0\in\g^*$ is a regular value of $\mu$, then $j$ is orbital (with respect to $\Vect(Z/G)$ and $\Vect(M/G)$).
\end{proposition}

\begin{proof}
By \hypref{t:regvalue}{Theorem} we know that in this case, orbits of $\ham(Z/G)$ and $\Vect(Z/G)$ coincide.  Thus, applying \hypref{p:jorbital2}{Proposition} we are done.
\end{proof}

\begin{remark}
Let $G$ be a compact Lie group acting on a connected manifold $M$, and let $Z$ be an invariant closed subset of $M$.  Then it is not true that the inclusion $j:Z/G\to M/G$ is orbital with respect to $\Vect(Z/G)$ and $\Vect(M/G)$.  Indeed, let $G=\SS^1$ and $M=\RR\times\RR^2$.  $G$ acts on $M$ diagonally, trivially on $\RR$ and by rotations on $\RR^2$.  The geometric quotient $M/G$ is diffeomorphic to the closed half-plane $\RR\times[0,\infty)$.  Set coordinates $(x,y,z)$ on $M$, where $x$ describes $\RR$ and $(y,z)$ describes $\RR^2$ in the product $\RR\times\RR^2$.  Let $Z$ be the cone in $M$ given by $\{x^2=y^2+z^2\}.$  This is $G$-invariant, and $Z/G$ is diffeomorphic to $\RR$.  $\partial_x$ is a vector field on $Z/G$ whose orbit is all of $Z/G$, yet $j(Z/G)$ is not contained in one orbit of $\Vect(M/G)$.
\end{remark}

Given the above remark, zero sets of momentum maps are still special invariant closed subsets.  We are left with the following question.

\begin{question}
Is $j$ orbital in the general case when $0\in\g^*$ is a critical value of $\mu$?
\end{question} 

\end{document}